\documentclass[12pt]{article}
\usepackage{type1cm}
\usepackage[utf8]{inputenc}
\usepackage[T1]{fontenc}
\usepackage{aecompl}
\usepackage{ae}
\usepackage{textcomp}
\usepackage[english]{babel}
\usepackage{tikz}
\usetikzlibrary{arrows}
\usepackage{pgffor}
\usepackage{subfigure}
\usepackage{hyperref}
\usepackage{hypbmsec}
\usepackage{amsmath,amsthm,amsfonts}
\usepackage{bbm}
\usepackage{enumerate}
\usepackage{dsfont}

\theoremstyle{plain}
\newtheorem{thm}{Theorem}[section]
\newtheorem{lem}[thm]{Lemma}
\newtheorem{defn}[thm]{Definition}

\newcommand{\e}{\mathrm{e}}
\newcommand{\E}{\mathbb{E}}
\renewcommand{\P}{\mathbb{P}}
\newcommand{\Z}{\mathbb{Z}}
\newcommand{\N}{\mathbb{N}}
\newcommand{\R}{\mathbb{R}}
\newcommand{\C}{\mathcal{C}_2}
\renewcommand{\O}{\mathcal{O}}
\renewcommand{\d}{\mathrm{d}}
\newcommand{\norm}[1]{\lVert #1 \rVert}
\newcommand{\abs}[1]{\lvert#1\rvert}
\newcommand{\indicator}[1]{\mathds{1}_{#1}}
\renewcommand{\labelenumi}{\upshape(\alph{enumi})}
\renewcommand{\theenumi}{(\alph{enumi})}
\newcommand{\p}{\dfrac{\partial}{\partial z}}
\newcommand{\pp}{\dfrac{\partial^2}{\partial z^2}}
\newcommand{\pr}[1]{\dfrac{\partial^{#1}}{\partial z^{#1}}}
\hypersetup{
  pdftitle = {Rotor-Router Aggregation on the Comb},
  pdfauthor = {Wilfried Huss, Ecaterina Sava},
  pdfkeywords = {growth model, comb, rotor-router, asymptotic shape, harmonic measure}
}

\makeatletter
\def\blfootnote{\xdef\@thefnmark{}\@footnotetext}
\makeatother
\newcommand{\arxivnotice}[1]
{\blfootnote{\vskip-3pt\noindent\fbox{\parbox{\linewidth}{\small#1}}}}

\title{Rotor-Router Aggregation on the Comb}
\title{Rotor-Router Aggregation on the Comb}
\author{Wilfried Huss\footnote{University of Siegen, Germany. Research supported by the FWF program FWF-P19115-N18}, 
Ecaterina Sava\footnote{Graz University of Technology, Austria. Research supported by the FWF program W 1230-N13}}

\begin{document}

\maketitle
\arxivnotice{This is an electronic reprint of the original article published in \textit{\href{www.combinatorics.org}{The Electronic Journal of Combinatorics}}, 2011, \textbf{18(1)}, P224. 
This reprint differs from the original in pagination and typographic detail.}
\begin{abstract}
We prove a shape theorem for rotor-router aggregation on the comb, for a specific initial
rotor configuration and clockwise rotor sequence for all vertices.
Furthermore, as an application of rotor-router walks, we describe the harmonic measure
of the rotor-router aggregate and related shapes, which is useful in the study of other
growth models on the comb. We also identify the shape for which the harmonic measure is uniform.
This gives the first known example where the rotor-router cluster has non-uniform harmonic measure,
and grows with different speeds in different directions.
\end{abstract}
\begin{center}
{\small\textbf{Keywords:} growth model, comb, rotor-router, asymptotic shape, harmonic measure.}
\end{center}
\section{Introduction}
Rotor-router walks are deterministic analogues to random walks,
which have been introduced into the physics literature under the
name \emph{Eulerian walks} by \textsc{Priezzhev, D.Dhar et al}
\cite{PhysRevLett.77.5079} as a model of \emph{self organized
criticality}, a concept established by \textsc{Bak, Tang and
Wiesenfeld} \cite{back_tang_wiesenfeld}.

In a rotor-router walk on a graph $G$, for each vertex $x\in G$ a cyclic
ordering $c(x)$ of its neighbours is chosen. At each vertex we have an
arrow (rotor) pointing to one of the neighbours of the vertex. A particle
performing a \emph{rotor-router walk} carries out the following procedure
at each step: first it changes the rotor at its current position $x$ to point
to the next neighbour of $x$ defined by the ordering $c(x)$, and then the 
particle moves to the neighbour the rotor is now pointing at.

The behaviour of rotor-router walks is in some respects remarkably
close to that of random walks. See for example \textsc{Cooper and Spencer}
\cite{cooper_spencer_2006} and \textsc{Doerr and Friedrich} \cite{doerr_friedrich_2007}.

In the present paper, we are interested in a process called \emph{rotor-router
aggregation}, defined as follows. Choose a root vertex $o\in G$ and let 
$R_1 = \lbrace o \rbrace$. The sets $R_n$ are defined recursively, by
\begin{equation*}
R_{n+1} = R_{n} \cup \lbrace z_n \rbrace \quad\text{ for } n \geq 1,
\end{equation*}
where $z_n$ is the first vertex outside $R_n$ that is visited
by a rotor-router walk started at the origin $o$.
The rotor configuration is not changed when a new particle is started
at the origin. We will call the set $R_n$ the \emph{rotor-router cluster}
of $n$ particles.

Rotor-router aggregation on the Euclidean lattice $\Z^d$ has been studied 
by \textsc{Levine and Peres} \cite{peres_levine_strong_spherical},
who showed that the rotor-router cluster $R_n$ forms
a ball in the usual Euclidean distance. On the homogeneous tree \textsc{Landau and Levine}
\cite{landau_levine_2009} proved that, under certain conditions on the initial
configuration of rotors, the rotor-router cluster $R_n$ forms a perfect ball with respect to the
graph metric, whenever it has the right amount of particles. \textsc{Kager and Levine}
\cite{kager_levine_rotor_aggregation} studied the shape of the rotor-router cluster on a
modified two-dimensional lattice, which they call the \emph{layered square lattice}.

In each of these examples the fluctuations of the cluster around the limiting shape are much
smaller in rotor-router aggregation than in the corresponding random growth model called
\emph{internal diffusion limited aggregation (IDLA)}, where particles perform independent
random walks before they settle and attach to the cluster. In the case of the homogeneous tree
and the layered square lattice, the fluctuations even vanish completely in the deterministic
model.

We will use the technique introduced in \cite{kager_levine_rotor_aggregation} in order to
study rotor-router aggregation on the two-dimensional comb $\C$, which is the spanning
tree of the two-dimensional lattice $\Z^2$, obtained by removing all horizontal
edges of $\Z^2$ except the ones on the $x$-axis.
In other words, the graph $\C$ can be constructed from a two-sided infinite path $\Z$
(the "\emph{backbone}" of the comb), by attaching copies of $\Z$ (the "\emph{teeth}") at
every vertex of the backbone.

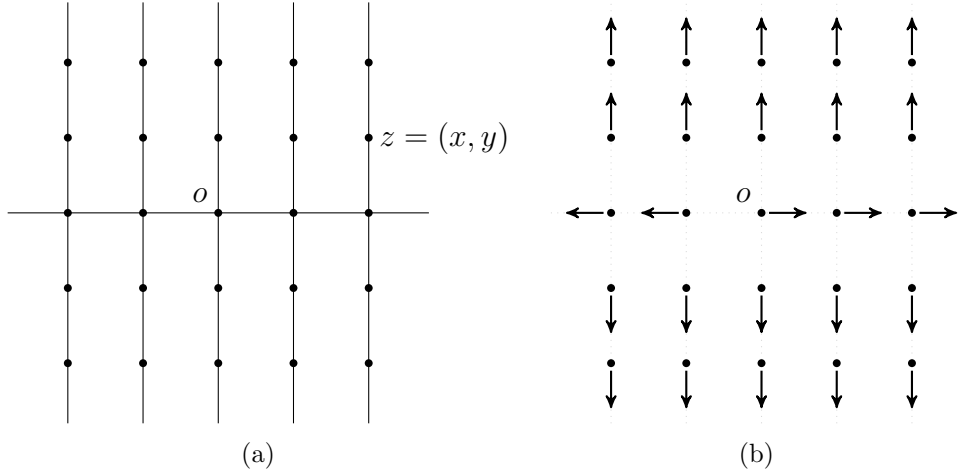
\begin{figure}
\centering
\subfigure[]
{
\begin{tikzpicture}
\foreach \x in {-2,...,2}
{
   \draw (\x, 2.8) -- (\x, -2.8);
   \foreach \y in {-2,...,2}
   \fill (\x,\y) circle (1.5pt);
}

\draw (-2.8, 0) -- (2.8, 0);

\coordinate[label=135:$o$] (O) at (0,0);
\coordinate[label=0:{$z=(x,y)$}] (Z) at (2, 1);
\end{tikzpicture}
\label{fig:comb}
}
\subfigure[]
{
\begin{tikzpicture}
\draw[black!15, dotted] (-2.8, 0) -- (2.8, 0);
\foreach \x in {-2,...,2}
{
    \draw[black!15, dotted] (\x, -2.8) -- (\x, 2.8);
    \foreach \y in {-2,...,2}
        \fill[black] (\x, \y) circle (1.5pt);
}

\coordinate[label=135:{$o$}] (origin) at (0,0);

\foreach \x in {-2,...,2}
{
    \foreach \y in {-2,...,-1}
        \draw[-stealth', line width=0.8pt, black] (\x, {\y-0.1}) -- (\x, {\y-0.6});

    \foreach \y in {1,...,2}
        \draw[-stealth', line width=0.8pt, black] (\x, {\y+0.1}) -- (\x, {\y+0.6});
}

\foreach \x in {-2,...,-1}
    \draw[-stealth', line width=0.8pt, black] ({\x-0.1}, 0) -- ({\x-0.6}, 0);

\foreach \x in {0,...,2}
    \draw[-stealth', line width=0.8pt, black] ({\x+0.1}, 0) -- ({\x+0.6}, 0);

\end{tikzpicture}
\label{initial_rotor_configuration}
}
\caption{\subref{fig:comb} The two-dimensional comb $\C$.
         \subref{initial_rotor_configuration} The initial rotor configuration $\rho_0$.}
\end{figure}

We use the standard embedding of the comb into the two-dimensional Euclidean lattice $\Z^2$,
and use Cartesian coordinates $z = (x,y)\in \Z^2$ to denote vertices of $\C$. The vertex $o = (0,0)$
will be the \emph{root vertex}, see Figure \ref{fig:comb}. For functions $g$ on the vertex set of $\C$ we
will often write $g(x,y)$ instead of $g(z)$, when $z = (x,y)$.

While $\C$ is a very simple graph, it has some remarkable properties. For example,
the \emph{Einstein relation} between the spectral-, walk- and fractal-dimension is
violated on the comb, see \textsc{Bertacchi} \cite{bertacchi_comb}. \textsc{Peres and Krishnapur}
\cite{peres_krishnapur_collide} showed that on $\C$ two independent simple random walks meet only
finitely often; this is the so-called \emph{finite collision property}.

The structure of this paper is as follows. In Section \ref{sec:rotor-router} we recall some
preliminary results due to \textsc{Kager and Levine} \cite{kager_levine_rotor_aggregation}, which 
will be applied in order to prove the main result of the paper.
In Section \ref{sec:initial_rotor_cfg} we describe the shape of the rotor-router cluster on $\C$, 
for the initial rotor configuration $\rho_0$ in Figure \ref{initial_rotor_configuration}. Define
\begin{equation}\label{eq:B_m}
B_m = \big\lbrace (x,y)\in\C:\: \abs{x} \leq m, \abs{y} \leq h(m-\abs{x})\big\rbrace\quad\text{for }m \in\N,
\end{equation}
for some function $h:\N_0\to\N_0$. The main result of the paper is the following.

\begin{thm}\label{thm:rotor_shape}
Let $R_n$ be the rotor-router cluster of $n$ particles on the comb $\C$, 
with initial rotor configuration as in Figure \ref{initial_rotor_configuration} 
and clockwise rotor sequence for all $x\in \C$. Let $B_m$ be as in \eqref{eq:B_m} with
\begin{equation*}
h(x) = \left\lfloor \frac{(x+1)^2}{3} \right\rfloor.
\end{equation*}
Then, for all $m\geq 0$ and $n_m = \abs{B_m}$, the rotor-router cluster $R_{n_m}$ satisfies
$R_{n_m} = B_m$.
\end{thm}

The main idea of the proof is to study rotor-router aggregation on the non-negative integers $\N_0$, 
and then to glue different copies of $\N_0$ on the \emph{``backbone''} of $\C$, in order 
to get information on the behaviour of the rotor-router cluster.
In the upcoming paper \cite{huss_sava_aggregation} the authors study IDLA on the comb.
This gives another case where all fluctuations disappear in rotor-router aggregation
when compared with IDLA.

As an application of rotor-router walks, in Section \ref{sec:harm_measure}
we give a method to describe the \emph{harmonic measure} of generic sets $B_m$, of the form
\eqref{eq:B_m}. For this, let
\begin{equation}\label{eq:boundary_Bm}
\partial B_m=\{z\in B_m:\ \exists \text{ neighbour } y \text{ of } z \text{ in }\C, \text{ such that }y\notin B_m \}
\end{equation}
be the \emph{inner boundary} of the set $B_m$.
The harmonic measure of $B_m$ is defined as the exit distribution of a \emph{simple random walk} 
from the set $B_m$. For $z\in\partial B_m$, we denote by $\nu_{m,o}(z)$ the probability that a 
simple random walk started at the origin $o\in \C$ hits $\partial B_m$ in $z$.
In order to compute $\nu_{m,o}(z)$, we consider a special rotor-router process, 
which allows us to obtain exact results in several cases. In particular,
we identify the subsets of the comb for which the harmonic measure is uniform. 

\begin{thm}
\label{thm:unif_harm_measure}
Let $B_m\subset \C$ be as in \eqref{eq:B_m}, with $h(x) = x^2$. Then 
the harmonic measure $\nu_{m,o}$ of $B_m$ is the uniform measure on $\partial B_m$.
\end{thm}

For the router-router cluster in Theorem \ref{thm:rotor_shape}, we are 
able to give asymptotics of the harmonic measure. We prove the following.

\begin{thm}
\label{lem:harmonic_measure}
Let $B_m\subset \C$ be as in \eqref{eq:B_m} with $h(x) = \left\lfloor (x+1)^2/3 \right\rfloor$,  
and $z=(z_x,z_y)\in \partial B_m$. There exists a function $e:\N_0\to\N$ with $\lim_{x\to\infty}\frac{e(x)}{x} = c$,
and $0 < c < 1/2$, such that for all $m\geq 0$ the harmonic measure $\nu_{m,o}(z)$ is proportional to $e(m - \abs{z_x})$.
\end{thm}
This gives the first example where the rotor-router cluster is not a set with
uniform harmonic measure, and grows faster in the vertical direction than in the
horizontal direction.

We want to emphasize that it is not easy to apply this method in many cases, 
since it requires exact knowledge of the odometer function of the rotor-router walk, 
and at least some insight in the structure of the \emph{Abelian sandpile group} of the
set under consideration. The connection of the \emph{Abelian sandpile group}
to the rotor-router model has been established in the physics
literature; see \cite{PhysRevE.58.5449, PhysRevLett.77.5079}. One can
define a group based on the action of a particle which performs a
rotor-router walk on the rotor configuration. This \emph{rotor-router group}
is abelian and isomorphic to the Abelian sandpile group. This
isomorphism has been proven formally in \cite{landau_levine_2009}.
For a self-contained introduction see the overview paper of \textsc{Holroyd, Levine, et.al.} \cite{chip_rotor_2008}.

\section{Preliminaries}
\label{sec:rotor-router}
Let $(G,E(G))$ be an infinite, undirected and connected graph,
with vertex set $G$, equipped with a symmetric \emph{adjacency relation}
$\sim$, which defines the set of edges $E(G)$ (as a subset of $G\times G$). We 
write $(x,y)$ for the edge between the pair of neighbours $x,y$. In order to simplify 
the notation, instead of writing $(G,E(G))$ for a graph, we shall write only
$G$, and it will be clear from the context whether we are considering edges or vertices.
We denote by $d(x)$ the degree of the vertex $x$, that is,
the number of neighbours of $x$ in $G$. Fix a nonempty subset $S\subset G$ of vertices 
called \emph{sinks}, and let $G'=G\setminus S$.

The \emph{odometer function} $u(x)$ of the 
rotor-router aggregation is defined as the number of particles sent out by 
the vertex $x$ during the creation of the rotor-router cluster $R_n$ of $n$ particles.

A \emph{rotor configuration} on $G$ is a function $\rho: G'\to G$, such that $\rho(x)$ is  
a neighbour of $x$, for all $x\in G'$, that is $(x,\rho(x))\in E(G)$. Hence, $\rho$
assigns to every vertex one of its neighbours.
A rotor configuration $\rho$ is called \emph{acyclic}, if the subgraph of $G$
spanned by the rotors contains no directed cycles. A \emph{particle configuration}
on $G$ is a function $\sigma: G\to \mathbb{Z}$, with finite support.
If $\sigma(x)=m>0$, we say that there are $m$ particles at vertex $x$.
The \emph{rotor sequence} at vertex $x$ will be denoted by
$c(x) = \big(x_0,x_1,\ldots,x_{d(x)-1}\big)$ where all $x_i\sim x$ and
$x_i\not=x_j$ for $i\not=j$, with $i,j=0,1,\ldots,d(x)-1$. 
If $y=x_i\in c(x)$, for some $i\in \{0,1,\ldots, d(x)-1\}$, we denote by
$y^+$ the vertex $x_{(i+1)\bmod\, d(x)}$.

\begin{defn}[Toppling operator]
\label{def:rotor_toppling_operator}
Fix a vertex $x\in G'$. For a rotor configuration $\rho$ and a particle configuration $\sigma$ on $G$, we define
the \emph{toppling operator} $F_x$, which sends one particle out of vertex $x$, by
\begin{equation*}
F_x(\rho,\sigma) = (\rho', \sigma'),
\end{equation*}
where the new rotor configuration $\rho'$ is given by
\begin{equation*}
\rho'(y) =
\begin{cases}
\rho(y)^+ \quad &\text{if } y=x, \\
\rho(y)         &\text{otherwise},
\end{cases}
\end{equation*}
and the new particle configuration $\sigma'$ is given by
\begin{equation*}
\sigma'(y) =
\begin{cases}
\sigma(y)-1 \quad &\text{if } y=x, \\
\sigma(y)+1       &\text{if } y=\rho'(x), \\
\sigma(y)         &\text{otherwise}.
\end{cases}
\end{equation*}
\end{defn}
So $F_x$ first changes the rotor configuration by rotating the arrow at $x$
to its next position in the cyclic ordering $c(x)$, and then it sends a particle to 
the vertex the rotor at $x$ is now pointing at. The operation $F_x$ of toppling at some vertex $x$
can be successful even if there is no particle at $x$. If this is the case,
then a ``virtual particle'' is sent away from $x$ and a ``hole'' is left there.
If there is already a hole at $x$, the operator $F_x$ will increase its depth
by one. In the normal rotor-router aggregation no holes are ever
created during the whole process. A sequence of topplings $\{x_k\}_{k\geq 1}$
is called \emph{legal}, if no holes are created when the vertices $x_k$ are
toppled in sequence.

Note that the toppling operators commute, i.e., $F_x F_y = F_y F_x$ for
all $x,y\in G'$. This is the usual \emph{abelian property} for rotor-router walks.
While the final configuration is always the same, rearranging the
order of the topplings can turn a legal toppling sequence into one that
creates holes and virtual particles.

Given a function $u: G' \to \N$, let
\begin{equation*}
F^u = \prod_{x\in G'} F_x^{u(x)},
\end{equation*}
where product means composition of the operators. Because of the abelian
property, $F^u$ is well defined.

In order to prove a shape result for rotor-router aggregation
on $\C$, for a specific initial configuration, we will apply a
stronger version of the usual Abelian property of rotor-router walks,
which has been recently introduced by \textsc{Kager and Levine}
\cite{kager_levine_rotor_aggregation}. We state it here for completeness.

\begin{thm}[Strong Abelian Property]
\label{thm:strong_abelian_property}
Let $\rho_0$ be a rotor configuration and $\sigma_0$ a particle configuration
on $G$. Given two functions $u_1, u_2: G' \to \N$, write
\begin{equation*}
F^{u_i}(\rho_0, \sigma_0) = (\rho_i, \sigma_i), \qquad i = 1,2.
\end{equation*}
If $\sigma_1 = \sigma_2$ on $G'$, and both $\rho_1$ and $\rho_2$ are acyclic,
then $u_1 = u_2$.
\end{thm}
Note that the equality $u_1=u_2$ implies also that $\rho_1=\rho_2$, and moreover 
$\sigma_1=\sigma_2$ on all of $G$. This result allows us to drop the hypothesis of
legality: each final particle configuration can only be achieved by
an unique amount of topplings for each vertex, even if we allow virtual
particles to be formed during the process.

\textsc{Friedrich and Levine} \cite{friedrich_levine} used the \emph{Strong Abelian Property}
 to give an exact characterization of the odometer function of rotor-router
aggregation. Recall that the odometer function $u(x)$ at some vertex $x$
represents the number of particles sent out by $x$ during the creation of the rotor-router
cluster.

\begin{thm}[Friedrich, Levine]
\label{thm:rotor_odometer_properties}
Let $G$ be a finite or infinite directed graph, $\rho_0$ an initial rotor
configuration on $G$, and $\sigma_0 = n\cdot\delta_o$. 
Fix $u_\star:G\to\N$, and let
\begin{equation*}
 A_\star = \big\{x\in G: u_\star(x) > 0\big\}.
\end{equation*}
Further define $\rho_\star$ and $\sigma_\star$ by
\begin{equation*}
F^{u_\star}(\rho_0,\sigma_0) = (\rho_\star, \sigma_\star).
\end{equation*}
Suppose the following properties hold
\begin{enumerate}[(a)]
\item $\sigma_\star \leq 1$,
\item $A_\star$ is finite,
\item $\sigma_\star(x) = 1$ for all $x\in A_\star$, and
\item  $\rho_\star$ is acyclic on $A_\star$.
\end{enumerate}
Then $u_\star$ is the rotor-router odometer function of $n$ particles.
\end{thm}
Using Theorem \ref{thm:rotor_odometer_properties}, for the proof of 
Theorem \ref{thm:rotor_shape} it is enough to give an explicit formula 
for the corresponding odometer function,
and to check if it satisfies all the required properties.

\section{Rotor-Router Aggregation}
\label{sec:initial_rotor_cfg}
Consider now the rotor-router aggregation on the comb $\C$, and the initial rotor
configuration as in Figure \ref{initial_rotor_configuration}. Through
this section, $B_m$ will be the set defined in \eqref{eq:B_m}, with
$h(x)$ given by
\begin{equation*}
 h(x) = \left\lfloor \frac{(x+1)^2}{3} \right\rfloor.
\end{equation*}

\begin{defn}
Let $(\rho, \sigma)$ be the final configuration of the rotor-router
aggregation process of $\abs{B_m}$ particles described in Theorem \ref{thm:rotor_shape}.
The configuration $(\rho, \sigma)$ is then called the $m$-th \emph{fully symmetric configuration}.
\end{defn}

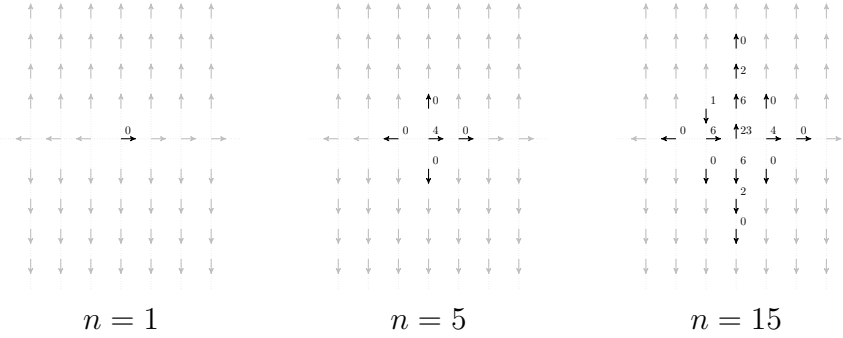
\begin{figure}
\centering
\begin{tikzpicture}
\draw[white] (-0.5, -0.5) rectangle (2.800000, 3.600000);\node (particles) at (1.200000,-0.800000) {$n=1$};
\pgflowlevelobj{\pgflowlevel{\pgftransformscale{0.400000}}}
{
\draw[black!15, dotted] (0, -1) -- (0, 9);
\draw[black!15, dotted] (1, -1) -- (1, 9);
\draw[black!15, dotted] (2, -1) -- (2, 9);
\draw[black!15, dotted] (3, -1) -- (3, 9);
\draw[black!15, dotted] (4, -1) -- (4, 9);
\draw[black!15, dotted] (5, -1) -- (5, 9);
\draw[black!15, dotted] (6, -1) -- (6, 9);
\draw[black!15, dotted] (-1, 4) -- (7, 4);
\draw[-stealth', line width=1pt, black!25] (0,0) -- (0.000000, -0.500000);
\draw[-stealth', line width=1pt, black!25] (0,1) -- (0.000000, 0.500000);
\draw[-stealth', line width=1pt, black!25] (0,2) -- (0.000000, 1.500000);
\draw[-stealth', line width=1pt, black!25] (0,3) -- (0.000000, 2.500000);
\draw[-stealth', line width=1pt, black!25] (0,4) -- (-0.500000, 4.000000);
\draw[-stealth', line width=1pt, black!25] (0,5) -- (0.000000, 5.500000);
\draw[-stealth', line width=1pt, black!25] (0,6) -- (0.000000, 6.500000);
\draw[-stealth', line width=1pt, black!25] (0,7) -- (0.000000, 7.500000);
\draw[-stealth', line width=1pt, black!25] (0,8) -- (0.000000, 8.500000);
\draw[-stealth', line width=1pt, black!25] (1,0) -- (1.000000, -0.500000);
\draw[-stealth', line width=1pt, black!25] (1,1) -- (1.000000, 0.500000);
\draw[-stealth', line width=1pt, black!25] (1,2) -- (1.000000, 1.500000);
\draw[-stealth', line width=1pt, black!25] (1,3) -- (1.000000, 2.500000);
\draw[-stealth', line width=1pt, black!25] (1,4) -- (0.500000, 4.000000);
\draw[-stealth', line width=1pt, black!25] (1,5) -- (1.000000, 5.500000);
\draw[-stealth', line width=1pt, black!25] (1,6) -- (1.000000, 6.500000);
\draw[-stealth', line width=1pt, black!25] (1,7) -- (1.000000, 7.500000);
\draw[-stealth', line width=1pt, black!25] (1,8) -- (1.000000, 8.500000);
\draw[-stealth', line width=1pt, black!25] (2,0) -- (2.000000, -0.500000);
\draw[-stealth', line width=1pt, black!25] (2,1) -- (2.000000, 0.500000);
\draw[-stealth', line width=1pt, black!25] (2,2) -- (2.000000, 1.500000);
\draw[-stealth', line width=1pt, black!25] (2,3) -- (2.000000, 2.500000);
\draw[-stealth', line width=1pt, black!25] (2,4) -- (1.500000, 4.000000);
\draw[-stealth', line width=1pt, black!25] (2,5) -- (2.000000, 5.500000);
\draw[-stealth', line width=1pt, black!25] (2,6) -- (2.000000, 6.500000);
\draw[-stealth', line width=1pt, black!25] (2,7) -- (2.000000, 7.500000);
\draw[-stealth', line width=1pt, black!25] (2,8) -- (2.000000, 8.500000);
\draw[-stealth', line width=1pt, black!25] (3,0) -- (3.000000, -0.500000);
\draw[-stealth', line width=1pt, black!25] (3,1) -- (3.000000, 0.500000);
\draw[-stealth', line width=1pt, black!25] (3,2) -- (3.000000, 1.500000);
\draw[-stealth', line width=1pt, black!25] (3,3) -- (3.000000, 2.500000);
\coordinate[label=45:\small{$0$}] (3_4) at (3,4);
\draw[-stealth', line width=1pt, black] (3,4) -- (3.500000, 4.000000);
\draw[-stealth', line width=1pt, black!25] (3,5) -- (3.000000, 5.500000);
\draw[-stealth', line width=1pt, black!25] (3,6) -- (3.000000, 6.500000);
\draw[-stealth', line width=1pt, black!25] (3,7) -- (3.000000, 7.500000);
\draw[-stealth', line width=1pt, black!25] (3,8) -- (3.000000, 8.500000);
\draw[-stealth', line width=1pt, black!25] (4,0) -- (4.000000, -0.500000);
\draw[-stealth', line width=1pt, black!25] (4,1) -- (4.000000, 0.500000);
\draw[-stealth', line width=1pt, black!25] (4,2) -- (4.000000, 1.500000);
\draw[-stealth', line width=1pt, black!25] (4,3) -- (4.000000, 2.500000);
\draw[-stealth', line width=1pt, black!25] (4,4) -- (4.500000, 4.000000);
\draw[-stealth', line width=1pt, black!25] (4,5) -- (4.000000, 5.500000);
\draw[-stealth', line width=1pt, black!25] (4,6) -- (4.000000, 6.500000);
\draw[-stealth', line width=1pt, black!25] (4,7) -- (4.000000, 7.500000);
\draw[-stealth', line width=1pt, black!25] (4,8) -- (4.000000, 8.500000);
\draw[-stealth', line width=1pt, black!25] (5,0) -- (5.000000, -0.500000);
\draw[-stealth', line width=1pt, black!25] (5,1) -- (5.000000, 0.500000);
\draw[-stealth', line width=1pt, black!25] (5,2) -- (5.000000, 1.500000);
\draw[-stealth', line width=1pt, black!25] (5,3) -- (5.000000, 2.500000);
\draw[-stealth', line width=1pt, black!25] (5,4) -- (5.500000, 4.000000);
\draw[-stealth', line width=1pt, black!25] (5,5) -- (5.000000, 5.500000);
\draw[-stealth', line width=1pt, black!25] (5,6) -- (5.000000, 6.500000);
\draw[-stealth', line width=1pt, black!25] (5,7) -- (5.000000, 7.500000);
\draw[-stealth', line width=1pt, black!25] (5,8) -- (5.000000, 8.500000);
\draw[-stealth', line width=1pt, black!25] (6,0) -- (6.000000, -0.500000);
\draw[-stealth', line width=1pt, black!25] (6,1) -- (6.000000, 0.500000);
\draw[-stealth', line width=1pt, black!25] (6,2) -- (6.000000, 1.500000);
\draw[-stealth', line width=1pt, black!25] (6,3) -- (6.000000, 2.500000);
\draw[-stealth', line width=1pt, black!25] (6,4) -- (6.500000, 4.000000);
\draw[-stealth', line width=1pt, black!25] (6,5) -- (6.000000, 5.500000);
\draw[-stealth', line width=1pt, black!25] (6,6) -- (6.000000, 6.500000);
\draw[-stealth', line width=1pt, black!25] (6,7) -- (6.000000, 7.500000);
\draw[-stealth', line width=1pt, black!25] (6,8) -- (6.000000, 8.500000);
}
\end{tikzpicture}
\hspace{0.5cm}
\begin{tikzpicture}
\draw[white] (-0.5, -0.5) rectangle (2.800000, 3.600000);\node (particles) at (1.200000,-0.800000) {$n=5$};
\pgflowlevelobj{\pgflowlevel{\pgftransformscale{0.400000}}}
{
\draw[black!15, dotted] (0, -1) -- (0, 9);
\draw[black!15, dotted] (1, -1) -- (1, 9);
\draw[black!15, dotted] (2, -1) -- (2, 9);
\draw[black!15, dotted] (3, -1) -- (3, 9);
\draw[black!15, dotted] (4, -1) -- (4, 9);
\draw[black!15, dotted] (5, -1) -- (5, 9);
\draw[black!15, dotted] (6, -1) -- (6, 9);
\draw[black!15, dotted] (-1, 4) -- (7, 4);
\draw[-stealth', line width=1pt, black!25] (0,0) -- (0.000000, -0.500000);
\draw[-stealth', line width=1pt, black!25] (0,1) -- (0.000000, 0.500000);
\draw[-stealth', line width=1pt, black!25] (0,2) -- (0.000000, 1.500000);
\draw[-stealth', line width=1pt, black!25] (0,3) -- (0.000000, 2.500000);
\draw[-stealth', line width=1pt, black!25] (0,4) -- (-0.500000, 4.000000);
\draw[-stealth', line width=1pt, black!25] (0,5) -- (0.000000, 5.500000);
\draw[-stealth', line width=1pt, black!25] (0,6) -- (0.000000, 6.500000);
\draw[-stealth', line width=1pt, black!25] (0,7) -- (0.000000, 7.500000);
\draw[-stealth', line width=1pt, black!25] (0,8) -- (0.000000, 8.500000);
\draw[-stealth', line width=1pt, black!25] (1,0) -- (1.000000, -0.500000);
\draw[-stealth', line width=1pt, black!25] (1,1) -- (1.000000, 0.500000);
\draw[-stealth', line width=1pt, black!25] (1,2) -- (1.000000, 1.500000);
\draw[-stealth', line width=1pt, black!25] (1,3) -- (1.000000, 2.500000);
\draw[-stealth', line width=1pt, black!25] (1,4) -- (0.500000, 4.000000);
\draw[-stealth', line width=1pt, black!25] (1,5) -- (1.000000, 5.500000);
\draw[-stealth', line width=1pt, black!25] (1,6) -- (1.000000, 6.500000);
\draw[-stealth', line width=1pt, black!25] (1,7) -- (1.000000, 7.500000);
\draw[-stealth', line width=1pt, black!25] (1,8) -- (1.000000, 8.500000);
\draw[-stealth', line width=1pt, black!25] (2,0) -- (2.000000, -0.500000);
\draw[-stealth', line width=1pt, black!25] (2,1) -- (2.000000, 0.500000);
\draw[-stealth', line width=1pt, black!25] (2,2) -- (2.000000, 1.500000);
\draw[-stealth', line width=1pt, black!25] (2,3) -- (2.000000, 2.500000);
\coordinate[label=45:\small{$0$}] (2_4) at (2,4);
\draw[-stealth', line width=1pt, black] (2,4) -- (1.500000, 4.000000);
\draw[-stealth', line width=1pt, black!25] (2,5) -- (2.000000, 5.500000);
\draw[-stealth', line width=1pt, black!25] (2,6) -- (2.000000, 6.500000);
\draw[-stealth', line width=1pt, black!25] (2,7) -- (2.000000, 7.500000);
\draw[-stealth', line width=1pt, black!25] (2,8) -- (2.000000, 8.500000);
\draw[-stealth', line width=1pt, black!25] (3,0) -- (3.000000, -0.500000);
\draw[-stealth', line width=1pt, black!25] (3,1) -- (3.000000, 0.500000);
\draw[-stealth', line width=1pt, black!25] (3,2) -- (3.000000, 1.500000);
\coordinate[label=45:\small{$0$}] (3_3) at (3,3);
\draw[-stealth', line width=1pt, black] (3,3) -- (3.000000, 2.500000);
\coordinate[label=45:\small{$4$}] (3_4) at (3,4);
\draw[-stealth', line width=1pt, black] (3,4) -- (3.500000, 4.000000);
\coordinate[label=45:\small{$0$}] (3_5) at (3,5);
\draw[-stealth', line width=1pt, black] (3,5) -- (3.000000, 5.500000);
\draw[-stealth', line width=1pt, black!25] (3,6) -- (3.000000, 6.500000);
\draw[-stealth', line width=1pt, black!25] (3,7) -- (3.000000, 7.500000);
\draw[-stealth', line width=1pt, black!25] (3,8) -- (3.000000, 8.500000);
\draw[-stealth', line width=1pt, black!25] (4,0) -- (4.000000, -0.500000);
\draw[-stealth', line width=1pt, black!25] (4,1) -- (4.000000, 0.500000);
\draw[-stealth', line width=1pt, black!25] (4,2) -- (4.000000, 1.500000);
\draw[-stealth', line width=1pt, black!25] (4,3) -- (4.000000, 2.500000);
\coordinate[label=45:\small{$0$}] (4_4) at (4,4);
\draw[-stealth', line width=1pt, black] (4,4) -- (4.500000, 4.000000);
\draw[-stealth', line width=1pt, black!25] (4,5) -- (4.000000, 5.500000);
\draw[-stealth', line width=1pt, black!25] (4,6) -- (4.000000, 6.500000);
\draw[-stealth', line width=1pt, black!25] (4,7) -- (4.000000, 7.500000);
\draw[-stealth', line width=1pt, black!25] (4,8) -- (4.000000, 8.500000);
\draw[-stealth', line width=1pt, black!25] (5,0) -- (5.000000, -0.500000);
\draw[-stealth', line width=1pt, black!25] (5,1) -- (5.000000, 0.500000);
\draw[-stealth', line width=1pt, black!25] (5,2) -- (5.000000, 1.500000);
\draw[-stealth', line width=1pt, black!25] (5,3) -- (5.000000, 2.500000);
\draw[-stealth', line width=1pt, black!25] (5,4) -- (5.500000, 4.000000);
\draw[-stealth', line width=1pt, black!25] (5,5) -- (5.000000, 5.500000);
\draw[-stealth', line width=1pt, black!25] (5,6) -- (5.000000, 6.500000);
\draw[-stealth', line width=1pt, black!25] (5,7) -- (5.000000, 7.500000);
\draw[-stealth', line width=1pt, black!25] (5,8) -- (5.000000, 8.500000);
\draw[-stealth', line width=1pt, black!25] (6,0) -- (6.000000, -0.500000);
\draw[-stealth', line width=1pt, black!25] (6,1) -- (6.000000, 0.500000);
\draw[-stealth', line width=1pt, black!25] (6,2) -- (6.000000, 1.500000);
\draw[-stealth', line width=1pt, black!25] (6,3) -- (6.000000, 2.500000);
\draw[-stealth', line width=1pt, black!25] (6,4) -- (6.500000, 4.000000);
\draw[-stealth', line width=1pt, black!25] (6,5) -- (6.000000, 5.500000);
\draw[-stealth', line width=1pt, black!25] (6,6) -- (6.000000, 6.500000);
\draw[-stealth', line width=1pt, black!25] (6,7) -- (6.000000, 7.500000);
\draw[-stealth', line width=1pt, black!25] (6,8) -- (6.000000, 8.500000);
}
\end{tikzpicture}
\hspace{0.5cm}
\begin{tikzpicture}
\draw[white] (-0.5, -0.5) rectangle (2.800000, 3.600000);\node (particles) at (1.200000,-0.800000) {$n=15$};
\pgflowlevelobj{\pgflowlevel{\pgftransformscale{0.400000}}}
{
\draw[black!15, dotted] (0, -1) -- (0, 9);
\draw[black!15, dotted] (1, -1) -- (1, 9);
\draw[black!15, dotted] (2, -1) -- (2, 9);
\draw[black!15, dotted] (3, -1) -- (3, 9);
\draw[black!15, dotted] (4, -1) -- (4, 9);
\draw[black!15, dotted] (5, -1) -- (5, 9);
\draw[black!15, dotted] (6, -1) -- (6, 9);
\draw[black!15, dotted] (-1, 4) -- (7, 4);
\draw[-stealth', line width=1pt, black!25] (0,0) -- (0.000000, -0.500000);
\draw[-stealth', line width=1pt, black!25] (0,1) -- (0.000000, 0.500000);
\draw[-stealth', line width=1pt, black!25] (0,2) -- (0.000000, 1.500000);
\draw[-stealth', line width=1pt, black!25] (0,3) -- (0.000000, 2.500000);
\draw[-stealth', line width=1pt, black!25] (0,4) -- (-0.500000, 4.000000);
\draw[-stealth', line width=1pt, black!25] (0,5) -- (0.000000, 5.500000);
\draw[-stealth', line width=1pt, black!25] (0,6) -- (0.000000, 6.500000);
\draw[-stealth', line width=1pt, black!25] (0,7) -- (0.000000, 7.500000);
\draw[-stealth', line width=1pt, black!25] (0,8) -- (0.000000, 8.500000);
\draw[-stealth', line width=1pt, black!25] (1,0) -- (1.000000, -0.500000);
\draw[-stealth', line width=1pt, black!25] (1,1) -- (1.000000, 0.500000);
\draw[-stealth', line width=1pt, black!25] (1,2) -- (1.000000, 1.500000);
\draw[-stealth', line width=1pt, black!25] (1,3) -- (1.000000, 2.500000);
\coordinate[label=45:\small{$0$}] (1_4) at (1,4);
\draw[-stealth', line width=1pt, black] (1,4) -- (0.500000, 4.000000);
\draw[-stealth', line width=1pt, black!25] (1,5) -- (1.000000, 5.500000);
\draw[-stealth', line width=1pt, black!25] (1,6) -- (1.000000, 6.500000);
\draw[-stealth', line width=1pt, black!25] (1,7) -- (1.000000, 7.500000);
\draw[-stealth', line width=1pt, black!25] (1,8) -- (1.000000, 8.500000);
\draw[-stealth', line width=1pt, black!25] (2,0) -- (2.000000, -0.500000);
\draw[-stealth', line width=1pt, black!25] (2,1) -- (2.000000, 0.500000);
\draw[-stealth', line width=1pt, black!25] (2,2) -- (2.000000, 1.500000);
\coordinate[label=45:\small{$0$}] (2_3) at (2,3);
\draw[-stealth', line width=1pt, black] (2,3) -- (2.000000, 2.500000);
\coordinate[label=45:\small{$6$}] (2_4) at (2,4);
\draw[-stealth', line width=1pt, black] (2,4) -- (2.500000, 4.000000);
\coordinate[label=45:\small{$1$}] (2_5) at (2,5);
\draw[-stealth', line width=1pt, black] (2,5) -- (2.000000, 4.500000);
\draw[-stealth', line width=1pt, black!25] (2,6) -- (2.000000, 6.500000);
\draw[-stealth', line width=1pt, black!25] (2,7) -- (2.000000, 7.500000);
\draw[-stealth', line width=1pt, black!25] (2,8) -- (2.000000, 8.500000);
\draw[-stealth', line width=1pt, black!25] (3,0) -- (3.000000, -0.500000);
\coordinate[label=45:\small{$0$}] (3_1) at (3,1);
\draw[-stealth', line width=1pt, black] (3,1) -- (3.000000, 0.500000);
\coordinate[label=45:\small{$2$}] (3_2) at (3,2);
\draw[-stealth', line width=1pt, black] (3,2) -- (3.000000, 1.500000);
\coordinate[label=45:\small{$6$}] (3_3) at (3,3);
\draw[-stealth', line width=1pt, black] (3,3) -- (3.000000, 2.500000);
\coordinate[label=45:\small{$23$}] (3_4) at (3,4);
\draw[-stealth', line width=1pt, black] (3,4) -- (3.000000, 4.500000);
\coordinate[label=45:\small{$6$}] (3_5) at (3,5);
\draw[-stealth', line width=1pt, black] (3,5) -- (3.000000, 5.500000);
\coordinate[label=45:\small{$2$}] (3_6) at (3,6);
\draw[-stealth', line width=1pt, black] (3,6) -- (3.000000, 6.500000);
\coordinate[label=45:\small{$0$}] (3_7) at (3,7);
\draw[-stealth', line width=1pt, black] (3,7) -- (3.000000, 7.500000);
\draw[-stealth', line width=1pt, black!25] (3,8) -- (3.000000, 8.500000);
\draw[-stealth', line width=1pt, black!25] (4,0) -- (4.000000, -0.500000);
\draw[-stealth', line width=1pt, black!25] (4,1) -- (4.000000, 0.500000);
\draw[-stealth', line width=1pt, black!25] (4,2) -- (4.000000, 1.500000);
\coordinate[label=45:\small{$0$}] (4_3) at (4,3);
\draw[-stealth', line width=1pt, black] (4,3) -- (4.000000, 2.500000);
\coordinate[label=45:\small{$4$}] (4_4) at (4,4);
\draw[-stealth', line width=1pt, black] (4,4) -- (4.500000, 4.000000);
\coordinate[label=45:\small{$0$}] (4_5) at (4,5);
\draw[-stealth', line width=1pt, black] (4,5) -- (4.000000, 5.500000);
\draw[-stealth', line width=1pt, black!25] (4,6) -- (4.000000, 6.500000);
\draw[-stealth', line width=1pt, black!25] (4,7) -- (4.000000, 7.500000);
\draw[-stealth', line width=1pt, black!25] (4,8) -- (4.000000, 8.500000);
\draw[-stealth', line width=1pt, black!25] (5,0) -- (5.000000, -0.500000);
\draw[-stealth', line width=1pt, black!25] (5,1) -- (5.000000, 0.500000);
\draw[-stealth', line width=1pt, black!25] (5,2) -- (5.000000, 1.500000);
\draw[-stealth', line width=1pt, black!25] (5,3) -- (5.000000, 2.500000);
\coordinate[label=45:\small{$0$}] (5_4) at (5,4);
\draw[-stealth', line width=1pt, black] (5,4) -- (5.500000, 4.000000);
\draw[-stealth', line width=1pt, black!25] (5,5) -- (5.000000, 5.500000);
\draw[-stealth', line width=1pt, black!25] (5,6) -- (5.000000, 6.500000);
\draw[-stealth', line width=1pt, black!25] (5,7) -- (5.000000, 7.500000);
\draw[-stealth', line width=1pt, black!25] (5,8) -- (5.000000, 8.500000);
\draw[-stealth', line width=1pt, black!25] (6,0) -- (6.000000, -0.500000);
\draw[-stealth', line width=1pt, black!25] (6,1) -- (6.000000, 0.500000);
\draw[-stealth', line width=1pt, black!25] (6,2) -- (6.000000, 1.500000);
\draw[-stealth', line width=1pt, black!25] (6,3) -- (6.000000, 2.500000);
\draw[-stealth', line width=1pt, black!25] (6,4) -- (6.500000, 4.000000);
\draw[-stealth', line width=1pt, black!25] (6,5) -- (6.000000, 5.500000);
\draw[-stealth', line width=1pt, black!25] (6,6) -- (6.000000, 6.500000);
\draw[-stealth', line width=1pt, black!25] (6,7) -- (6.000000, 7.500000);
\draw[-stealth', line width=1pt, black!25] (6,8) -- (6.000000, 8.500000);
}
\end{tikzpicture}
\caption{\label{fig:symmetric_config_3}
The first three fully symmetric configurations, consisting of $n$ particles. The numbers 
on the arrows are the values of the odometer function $u_n$.}
\end{figure}

Figures \ref{fig:symmetric_config_3} and \ref{fig:symmetric_config_6_7}
show examples of fully symmetric configurations.
In Figure \ref{fig:symmetric_config_6_7}, one can observe that the fully symmetric
configuration of $\abs{B_7}$ particles, as well as its corresponding rotor-router
odometer function, are obtained by shifting ``half`` of the configuration of
$\abs{B_6}$ particles one step in the direction of the positive resp. negative
$x$-axis and filling in the values for the ''tooth`` corresponding to $x=0$.
It turns out that this is true for all fully symmetric configurations (with the
exception of the first 3). This property will play an important role in the 
proof of Theorem \ref{thm:rotor_shape}.

For proving Theorem \ref{thm:rotor_shape}, an exact
expression for the cardinality of the sets $B_m$ is needed.

\begin{lem}
\label{prop:cardinality_B_m}
Let $B_m$ be the set defined in \eqref{eq:B_m}, with $h(x) = \left\lfloor \frac{(x+1)^2}{3} \right\rfloor$. 
Then, for all $m\geq 0$ the cardinality of $B_m$ is given by
\begin{equation}
\label{eq:B_m_size}
\abs{B_m} = \frac{1}{9}\big[4 m^3 + 12 m^2 + 24 m + 5 + 2 \big((m+2)\bmod{3}\big)\big].
\end{equation}
\end{lem}
\begin{proof}
In order to simplify the statement of the Lemma, we have to distinguish 
three cases, namely for $m=3k+i$, with $i=0,1,2$. The right-hand side of 
\eqref{eq:B_m_size} is then equal to
\begin{align}
\label{eq:B_m_mod3}
\begin{split}
N_0(k)& = 12k^{3} + 12k^{2} + 8k + 1, \text{ for }m=3k\\
N_1(k)& = 12k^{3} + 24k^{2} + 20k + 5, \text{ for }m=3k+1\\
N_2(k)& = 12k^{3} + 36k^{2} + 40k + 15, \text{ for }m=3k+2.
\end{split}
\end{align}
Moreover, in the three cases $m=3k+i$, with $i=0,1,2$, the function $h$
can be written as follows.
\begin{align}
\label{eq:h_m_mod3}
\begin{split}
h(3k)   & =3k^2+2k\\
h(3k+1) & =3k^2+4k+1\\
h(3k+2) & =3k^2+6k+3.
\end{split}
\end{align}
We prove \eqref{eq:B_m_mod3} by induction on $k$. The base case $m=1$
is immediate from the definition of $B_m$. With
\begin{equation*}
\abs{B_{m+1}}=\abs{B_{m}}+2\big[h(m)+h(m+1)+1\big]
\end{equation*}
follows the inductive step.
\end{proof}
Next, we will find an exact formula for the odometer function of the 
rotor-router aggregation defined in Theorem \ref{thm:rotor_shape}, 
and we shall prove its correctness using Theorem \ref{thm:rotor_odometer_properties}.
For this, we first have a detailed look at the rotor-router process
on the non-negative integers.

\begin{figure}
\centering
\hspace{0.5cm}

\hspace{0.5cm}
\caption{\label{fig:symmetric_config_6_7}
The $6^\text{th}$ and $7^\text{th}$ fully symmetric configurations,
consisting of $\abs{B_6}$ and $\abs{B_7}$ particles. The numbers are the values of the
odometer function $u_6$ and $u_7$, respectively.}
\end{figure}

\subsection{Rotor-Router on the non-negative Integers}
\label{sec:rotor_router_on_integers}
For a better understanding of the rotor-router process on the comb  $\C$, 
we first analyse it on the half-line, where it is very simple. 
Consider $G = \N_0$, with sink vertex $0$, and the initial
rotor configuration $\tilde{\rho}_0:\N \to \N_0$ given by
\begin{equation*}
 \tilde{\rho}_0(y) = y+1,\quad \text{for all } y\in \N.
\end{equation*}
Let $\tilde{R}_{1}=\{1\}$, and define a modified rotor-router 
aggregation process $\tilde{R}_{n}$ recursively as follows. Start a rotor-router walk in $1$, and stop the
particle when it either reaches the sink $0$, or exits the previous cluster
$\tilde{R}_{n-1}$. Denote by $\tilde{z}_n$ the vertex where the $n$-th particle stops,
and by $\tilde{\rho}_n$ and $\tilde{u}_n$ the rotor configuration and odometer function at that
time. Then,
\begin{equation*}
\tilde{R}_{n}=
\begin{cases}
\tilde{R}_{n-1}\cup \{\tilde{z}_n\} , & \text{if } \tilde{z}_n\not=0 \\
\tilde{R}_{n-1} , & \text{otherwise}.
\end{cases}
\end{equation*}
Obviously $\tilde{R}_{n} = \{1,\ldots,\tilde h(n)\}$ for some sequence $\tilde h(n)$. Since $\tilde{\rho}_0$ 
is acyclic, all rotor configurations $\tilde{\rho}_n$ are acyclic and have the form
\begin{equation*}
\tilde{\rho}_n(y)= \begin{cases}
            y-1,\quad & 0\leq y\leq \tilde r(n)\\
            y+1,\quad & \text{otherwise},
           \end{cases}
\end{equation*}
for some numbers $0 \leq \tilde r(n) \leq  \tilde h(n)$. Here, $\tilde r(n)$ represents the vertex 
where the rotors change direction: all rotors up to $\tilde r(n)$ point inwards 
($\downarrow$), and all rotors from $\tilde r(n)+1$ up to $\tilde h(n)$ point outwards ($\uparrow$).

For numbers $h, r$ and $y$ in $\N_0$, with $0 \leq r \leq  h$ 
define the function $\tilde u$ as
\begin{equation}\label{eq:odometer_generic}
 \tilde{u}(h, r, y)=\begin{cases}
          f(h-y)+e(r-y),\quad & 1\leq y\leq r\\
          f(h-y),               & r < y\leq h\\
          0,                      &\text{otherwise},
         \end{cases} 
\end{equation}
where the functions $e$ and $f$ are given by $e(y)=2y+1$ and $f(y)=y(y+1)$.

The odometer function $\tilde{u}(n)$ of the rotor-router process $\tilde R_n$
can now be defined in terms of \eqref{eq:odometer_generic} by setting
\begin{align}
\label{eq:tilde_hr}
\tilde{h}(n) = \max\left\lbrace k\in\N:\: \frac{k(k+1)}{2} \leq n\right\rbrace, \qquad
\tilde{r}(n) = n - \frac{\tilde{h}(n)\big(\tilde{h}(n)+1\big)}{2}
\end{align}
and
\begin{equation}\label{eq:odometer_z}
\tilde{u}_n(y) = \tilde{u}\big(\tilde h(n), \tilde r(n),y\big),
\end{equation}
for $\tilde{h}(n)$ and $\tilde{r}(n)$ as defined in \eqref{eq:tilde_hr}.
It is easy to verify by induction that $\tilde{u}_n$ correctly describes the 
odometer function of $\tilde{R}_n$. See Figure \ref{fig:rotor_z} for a
graphical representation of the process $\tilde{R}_n$. 

\begin{figure}[t]
\centering
\begin{tikzpicture}
\draw[white] (-0.7, -0.5) rectangle (9.000000, 4.020000);\pgflowlevelobj{\pgflowlevel{\pgftransformscale{0.600000}}}
{
\draw[-stealth', black] (-0.5, -0.5) -- (-0.5, 6);
\coordinate[label=90:\Large{$\N_0$}](zz) at (-0.5, 6);
\coordinate[label=180:$0$] (y0) at (-0.5, 0);
\draw (-0.5, 0) -- (-0.52, 0);
\coordinate[label=180:$1$] (y1) at (-0.5, 1);
\draw (-0.5, 1) -- (-0.52, 1);
\coordinate[label=180:$2$] (y2) at (-0.5, 2);
\draw (-0.5, 2) -- (-0.52, 2);
\coordinate[label=180:$3$] (y3) at (-0.5, 3);
\draw (-0.5, 3) -- (-0.52, 3);
\coordinate[label=180:$4$] (y4) at (-0.5, 4);
\draw (-0.5, 4) -- (-0.52, 4);
\coordinate[label=180:$5$] (y5) at (-0.5, 5);
\draw (-0.5, 5) -- (-0.52, 5);

\draw[-stealth', black] (-0.5, -0.5) -- (15-0.5, -0.5);
\coordinate[label=0:\Large{$n$}](zz) at (15-0.5, -0.5);

\foreach \x in {1,...,15}
{
  \coordinate[label=-90:$\x$] (x\x) at (\x-1, -0.6);
  \draw (\x-1,-0.5) -- (\x-1, -0.52);
}

\draw[black!15, dotted] (0, 0) -- (0, 6-0.25);
\fill[black!15] (0, 0) circle (2pt);
\coordinate[label=45:\small{$0$}] (0_1) at (0,1);
\draw[-stealth', line width=1pt, black] (0,1) -- (0.000000, 1.500000);
\draw[-stealth', line width=1pt, black!25] (0,2) -- (0.000000, 2.500000);
\draw[-stealth', line width=1pt, black!25] (0,3) -- (0.000000, 3.500000);
\draw[-stealth', line width=1pt, black!25] (0,4) -- (0.000000, 4.500000);
\draw[-stealth', line width=1pt, black!25] (0,5) -- (0.000000, 5.500000);
\fill[red] (0, 1) circle (3pt);
\draw[black!15, dotted] (1, 0) -- (1, 6-0.25);
\fill[black!15] (1, 0) circle (2pt);
\coordinate[label=45:\small{$1$}] (1_1) at (1,1);
\draw[-stealth', line width=1pt, black] (1,1) -- (1.000000, 0.500000);
\draw[-stealth', line width=1pt, black!25] (1,2) -- (1.000000, 2.500000);
\draw[-stealth', line width=1pt, black!25] (1,3) -- (1.000000, 3.500000);
\draw[-stealth', line width=1pt, black!25] (1,4) -- (1.000000, 4.500000);
\draw[-stealth', line width=1pt, black!25] (1,5) -- (1.000000, 5.500000);
\fill[red] (1, 0) circle (3pt);
\draw[black!15, dotted] (2, 0) -- (2, 6-0.25);
\fill[black!15] (2, 0) circle (2pt);
\coordinate[label=45:\small{$2$}] (2_1) at (2,1);
\draw[-stealth', line width=1pt, black] (2,1) -- (2.000000, 1.500000);
\coordinate[label=45:\small{$0$}] (2_2) at (2,2);
\draw[-stealth', line width=1pt, black] (2,2) -- (2.000000, 2.500000);
\draw[-stealth', line width=1pt, black!25] (2,3) -- (2.000000, 3.500000);
\draw[-stealth', line width=1pt, black!25] (2,4) -- (2.000000, 4.500000);
\draw[-stealth', line width=1pt, black!25] (2,5) -- (2.000000, 5.500000);
\fill[red] (2, 2) circle (3pt);
\draw[black!15, dotted] (3, 0) -- (3, 6-0.25);
\fill[black!15] (3, 0) circle (2pt);
\coordinate[label=45:\small{$3$}] (3_1) at (3,1);
\draw[-stealth', line width=1pt, black] (3,1) -- (3.000000, 0.500000);
\coordinate[label=45:\small{$0$}] (3_2) at (3,2);
\draw[-stealth', line width=1pt, black] (3,2) -- (3.000000, 2.500000);
\draw[-stealth', line width=1pt, black!25] (3,3) -- (3.000000, 3.500000);
\draw[-stealth', line width=1pt, black!25] (3,4) -- (3.000000, 4.500000);
\draw[-stealth', line width=1pt, black!25] (3,5) -- (3.000000, 5.500000);
\fill[red] (3, 0) circle (3pt);
\draw[black!15, dotted] (4, 0) -- (4, 6-0.25);
\fill[black!15] (4, 0) circle (2pt);
\coordinate[label=45:\small{$5$}] (4_1) at (4,1);
\draw[-stealth', line width=1pt, black] (4,1) -- (4.000000, 0.500000);
\coordinate[label=45:\small{$1$}] (4_2) at (4,2);
\draw[-stealth', line width=1pt, black] (4,2) -- (4.000000, 1.500000);
\draw[-stealth', line width=1pt, black!25] (4,3) -- (4.000000, 3.500000);
\draw[-stealth', line width=1pt, black!25] (4,4) -- (4.000000, 4.500000);
\draw[-stealth', line width=1pt, black!25] (4,5) -- (4.000000, 5.500000);
\fill[red] (4, 0) circle (3pt);
\draw[black!15, dotted] (5, 0) -- (5, 6-0.25);
\fill[black!15] (5, 0) circle (2pt);
\coordinate[label=45:\small{$6$}] (5_1) at (5,1);
\draw[-stealth', line width=1pt, black] (5,1) -- (5.000000, 1.500000);
\coordinate[label=45:\small{$2$}] (5_2) at (5,2);
\draw[-stealth', line width=1pt, black] (5,2) -- (5.000000, 2.500000);
\coordinate[label=45:\small{$0$}] (5_3) at (5,3);
\draw[-stealth', line width=1pt, black] (5,3) -- (5.000000, 3.500000);
\draw[-stealth', line width=1pt, black!25] (5,4) -- (5.000000, 4.500000);
\draw[-stealth', line width=1pt, black!25] (5,5) -- (5.000000, 5.500000);
\fill[red] (5, 3) circle (3pt);
\draw[black!15, dotted] (6, 0) -- (6, 6-0.25);
\fill[black!15] (6, 0) circle (2pt);
\coordinate[label=45:\small{$7$}] (6_1) at (6,1);
\draw[-stealth', line width=1pt, black] (6,1) -- (6.000000, 0.500000);
\coordinate[label=45:\small{$2$}] (6_2) at (6,2);
\draw[-stealth', line width=1pt, black] (6,2) -- (6.000000, 2.500000);
\coordinate[label=45:\small{$0$}] (6_3) at (6,3);
\draw[-stealth', line width=1pt, black] (6,3) -- (6.000000, 3.500000);
\draw[-stealth', line width=1pt, black!25] (6,4) -- (6.000000, 4.500000);
\draw[-stealth', line width=1pt, black!25] (6,5) -- (6.000000, 5.500000);
\fill[red] (6, 0) circle (3pt);
\draw[black!15, dotted] (7, 0) -- (7, 6-0.25);
\fill[black!15] (7, 0) circle (2pt);
\coordinate[label=45:\small{$9$}] (7_1) at (7,1);
\draw[-stealth', line width=1pt, black] (7,1) -- (7.000000, 0.500000);
\coordinate[label=45:\small{$3$}] (7_2) at (7,2);
\draw[-stealth', line width=1pt, black] (7,2) -- (7.000000, 1.500000);
\coordinate[label=45:\small{$0$}] (7_3) at (7,3);
\draw[-stealth', line width=1pt, black] (7,3) -- (7.000000, 3.500000);
\draw[-stealth', line width=1pt, black!25] (7,4) -- (7.000000, 4.500000);
\draw[-stealth', line width=1pt, black!25] (7,5) -- (7.000000, 5.500000);
\fill[red] (7, 0) circle (3pt);
\draw[black!15, dotted] (8, 0) -- (8, 6-0.25);
\fill[black!15] (8, 0) circle (2pt);
\coordinate[label=45:\small{$11$}] (8_1) at (8,1);
\draw[-stealth', line width=1pt, black] (8,1) -- (8.000000, 0.500000);
\coordinate[label=45:\small{$5$}] (8_2) at (8,2);
\draw[-stealth', line width=1pt, black] (8,2) -- (8.000000, 1.500000);
\coordinate[label=45:\small{$1$}] (8_3) at (8,3);
\draw[-stealth', line width=1pt, black] (8,3) -- (8.000000, 2.500000);
\draw[-stealth', line width=1pt, black!25] (8,4) -- (8.000000, 4.500000);
\draw[-stealth', line width=1pt, black!25] (8,5) -- (8.000000, 5.500000);
\fill[red] (8, 0) circle (3pt);
\draw[black!15, dotted] (9, 0) -- (9, 6-0.25);
\fill[black!15] (9, 0) circle (2pt);
\coordinate[label=45:\small{$12$}] (9_1) at (9,1);
\draw[-stealth', line width=1pt, black] (9,1) -- (9.000000, 1.500000);
\coordinate[label=45:\small{$6$}] (9_2) at (9,2);
\draw[-stealth', line width=1pt, black] (9,2) -- (9.000000, 2.500000);
\coordinate[label=45:\small{$2$}] (9_3) at (9,3);
\draw[-stealth', line width=1pt, black] (9,3) -- (9.000000, 3.500000);
\coordinate[label=45:\small{$0$}] (9_4) at (9,4);
\draw[-stealth', line width=1pt, black] (9,4) -- (9.000000, 4.500000);
\draw[-stealth', line width=1pt, black!25] (9,5) -- (9.000000, 5.500000);
\fill[red] (9, 4) circle (3pt);
\draw[black!15, dotted] (10, 0) -- (10, 6-0.25);
\fill[black!15] (10, 0) circle (2pt);
\coordinate[label=45:\small{$13$}] (10_1) at (10,1);
\draw[-stealth', line width=1pt, black] (10,1) -- (10.000000, 0.500000);
\coordinate[label=45:\small{$6$}] (10_2) at (10,2);
\draw[-stealth', line width=1pt, black] (10,2) -- (10.000000, 2.500000);
\coordinate[label=45:\small{$2$}] (10_3) at (10,3);
\draw[-stealth', line width=1pt, black] (10,3) -- (10.000000, 3.500000);
\coordinate[label=45:\small{$0$}] (10_4) at (10,4);
\draw[-stealth', line width=1pt, black] (10,4) -- (10.000000, 4.500000);
\draw[-stealth', line width=1pt, black!25] (10,5) -- (10.000000, 5.500000);
\fill[red] (10, 0) circle (3pt);
\draw[black!15, dotted] (11, 0) -- (11, 6-0.25);
\fill[black!15] (11, 0) circle (2pt);
\coordinate[label=45:\small{$15$}] (11_1) at (11,1);
\draw[-stealth', line width=1pt, black] (11,1) -- (11.000000, 0.500000);
\coordinate[label=45:\small{$7$}] (11_2) at (11,2);
\draw[-stealth', line width=1pt, black] (11,2) -- (11.000000, 1.500000);
\coordinate[label=45:\small{$2$}] (11_3) at (11,3);
\draw[-stealth', line width=1pt, black] (11,3) -- (11.000000, 3.500000);
\coordinate[label=45:\small{$0$}] (11_4) at (11,4);
\draw[-stealth', line width=1pt, black] (11,4) -- (11.000000, 4.500000);
\draw[-stealth', line width=1pt, black!25] (11,5) -- (11.000000, 5.500000);
\fill[red] (11, 0) circle (3pt);
\draw[black!15, dotted] (12, 0) -- (12, 6-0.25);
\fill[black!15] (12, 0) circle (2pt);
\coordinate[label=45:\small{$17$}] (12_1) at (12,1);
\draw[-stealth', line width=1pt, black] (12,1) -- (12.000000, 0.500000);
\coordinate[label=45:\small{$9$}] (12_2) at (12,2);
\draw[-stealth', line width=1pt, black] (12,2) -- (12.000000, 1.500000);
\coordinate[label=45:\small{$3$}] (12_3) at (12,3);
\draw[-stealth', line width=1pt, black] (12,3) -- (12.000000, 2.500000);
\coordinate[label=45:\small{$0$}] (12_4) at (12,4);
\draw[-stealth', line width=1pt, black] (12,4) -- (12.000000, 4.500000);
\draw[-stealth', line width=1pt, black!25] (12,5) -- (12.000000, 5.500000);
\fill[red] (12, 0) circle (3pt);
\draw[black!15, dotted] (13, 0) -- (13, 6-0.25);
\fill[black!15] (13, 0) circle (2pt);
\coordinate[label=45:\small{$19$}] (13_1) at (13,1);
\draw[-stealth', line width=1pt, black] (13,1) -- (13.000000, 0.500000);
\coordinate[label=45:\small{$11$}] (13_2) at (13,2);
\draw[-stealth', line width=1pt, black] (13,2) -- (13.000000, 1.500000);
\coordinate[label=45:\small{$5$}] (13_3) at (13,3);
\draw[-stealth', line width=1pt, black] (13,3) -- (13.000000, 2.500000);
\coordinate[label=45:\small{$1$}] (13_4) at (13,4);
\draw[-stealth', line width=1pt, black] (13,4) -- (13.000000, 3.500000);
\draw[-stealth', line width=1pt, black!25] (13,5) -- (13.000000, 5.500000);
\fill[red] (13, 0) circle (3pt);
\draw[black!15, dotted] (14, 0) -- (14, 6-0.25);
\fill[black!15] (14, 0) circle (2pt);
\coordinate[label=45:\small{$20$}] (14_1) at (14,1);
\draw[-stealth', line width=1pt, black] (14,1) -- (14.000000, 1.500000);
\coordinate[label=45:\small{$12$}] (14_2) at (14,2);
\draw[-stealth', line width=1pt, black] (14,2) -- (14.000000, 2.500000);
\coordinate[label=45:\small{$6$}] (14_3) at (14,3);
\draw[-stealth', line width=1pt, black] (14,3) -- (14.000000, 3.500000);
\coordinate[label=45:\small{$2$}] (14_4) at (14,4);
\draw[-stealth', line width=1pt, black] (14,4) -- (14.000000, 4.500000);
\coordinate[label=45:\small{$0$}] (14_5) at (14,5);
\draw[-stealth', line width=1pt, black] (14,5) -- (14.000000, 5.500000);
\fill[red] (14, 5) circle (3pt);
}
\end{tikzpicture}
\caption{\label{fig:rotor_z}The first steps of the process $\tilde{R}_n$ on $\N$.
The dots mark the vertex where the current particle stopped.} 
\end{figure}
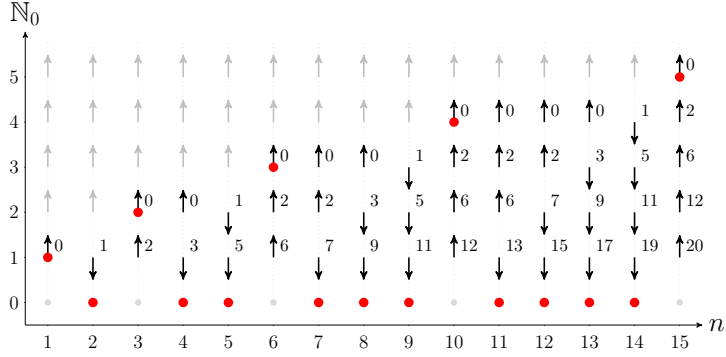

\subsection{Rotor-Router on the Comb}
Since the rotor-router aggregation on the ``half-teeth'' of $\C$ behaves like the process
$\tilde{R}_n$ from the previous section, it is enough to determine the numbers $h$ and $r$
in \eqref{eq:odometer_generic}, which now depend on $x$ and the number of particles, in order to fully
specify the odometer function on $\C$ for points off the $x$-axis.

Let $B_m$ as defined in \eqref{eq:B_m}, with $h(x) = \left\lfloor\frac{(x+1)^2}{3}\right\rfloor$ 
and define $r(x)$ by
\begin{equation}
\label{eq:def_r_x}
 r(x) =
\begin{cases}
 0, & x \in\{0,1\} \\
\frac{1}{18}\bigl(x^2-7x+10\bigr), & x \equiv 2\bmod{3} \\
\frac{1}{6}\bigl(x^2-x+6), \quad& \text{otherwise}.
\end{cases}
\end{equation}
Define $u_m:B_m\to \N$ by
\begin{equation}
\label{eq:comb_rotor_odometer}
 u_m(x,y) = u'(m-\abs{x},\abs{y}),
\end{equation}
 where
\begin{equation}\label{eq:comb_odometer_shifted}
 u'(x,y) =
\begin{cases}
  \tilde{u}\bigl(h(x),r(x),y\bigr), & y > 0 \\
  2 f\big( h(x) \big)+2 e\big(r(x)\big) - 2 - \indicator{\{x=2\}}, & y = 0,
\end{cases}
\end{equation}
with $\tilde{u}$ as in \eqref{eq:odometer_generic}, $h(x)$ and $r(x)$ defined as above,
and $e(x)=2x+1$ and $f(x)=x(x+1)$. We claim that $u_m$ is the odometer function 
for rotor-router aggregation of $\abs{B_m}$ particles on the comb.

\subsubsection{Proof of Theorem \ref{thm:rotor_shape}}
Let $B_m$ as defined in \eqref{eq:B_m}, with $h(x) = \lfloor(x+1)^2/3\rfloor$
as in Theorem \ref{thm:rotor_shape}, and $u_m$ defined as in \eqref{eq:comb_rotor_odometer}.

\begin{lem}
Let $\rho_0$ be the initial rotor configuration defined in Figure
\ref{initial_rotor_configuration}, and the initial particle
configuration $\sigma_0 = \abs{B_m}\cdot\delta_o$.
Furthermore, define $\rho_m$ and $\sigma_m$ as
\begin{equation*}
(\rho_m,\sigma_m) = F^{u_m}(\rho_0,\sigma_0),
\end{equation*}
with $u_m$ as in \eqref{eq:comb_rotor_odometer}. If a clockwise rotor sequence
is assumed for all vertices, then $u_m$ is the odometer function of the rotor-router 
aggregation with $\abs{B_m}$ particles and moreover $\sigma_m = \indicator{B_m}$, for all $m \geq 3$.
\end{lem}

\begin{proof}
To verify that $u_m$ is indeed the odometer function of this rotor-router process, 
we need to check the four properties of Theorem \ref{thm:rotor_odometer_properties}, with
\begin{equation*}
A_\star = B_m \setminus \partial B_m,
\end{equation*}
where $\partial B_m$ is the inner boundary of $B_m$ defined in \eqref{eq:boundary_Bm}.
The set $A_\star$ is obviously finite. For those vertices
$z\in A_\star$ that have neighbours in $B_m\setminus A_\star$, we have 
by \eqref{eq:odometer_z} and \eqref{eq:odometer_generic} that $u_m(z) \leq 3$ if $z$ 
is not on the $x$-axis, and $u_m(z) = 4$ otherwise. In both cases at most one particle is 
sent to some vertex outside of $A_\star$, hence $\sigma_m(z) \leq 1$ for all $z\not\in A_\star$.

Next we verify that the final particle configuration $\sigma_m$
is equal to $1$ on $ A_\star$. Since by definition $u_m$ is symmetric, it is
enough to consider only one quadrant. Additionally, we shift the coordinate system 
such that the point $(-m, 0)$ lies at the origin, which means that we can 
work with the function $u'$, defined in \eqref{eq:comb_odometer_shifted}. Since $u'$ does 
not depend on the parameter $m$, most of what follows holds independently of $m$. 
Only for the center $(m,0)$ of the set $B_m$ (Case 4), we need to take the parameter 
$m$ into account. Let $z=(x,y)\in A_\star$ with $x,y\geq 0$. We distinguish several cases.

%==================================================================================
%==============The case y>2========================================================
\textbf{Case 1. $y\geq 2$}: For vertices $(x,y)$, with $y\geq 2$, the rotor-router aggregation behaves
as the process $\tilde{R}_n$, defined in Section \ref{sec:rotor_router_on_integers}.
From Figure \ref{fig:rotor_z}, and due to the fact that the final rotor configuration restricted to 
each ``tooth'' is acyclic, there are again four possible situations:
\begin{enumerate}[(a)]
\item
The rotors at the vertices $(x, y-1)$, $(x,y)$ and $(x,y+1)$ all point outwards
($\uparrow$). This is the case when $r(x)<y-1$, hence the vertex $(x,y)$ receives 
$\frac{1}{2}\tilde{u}_m(y-1)+\frac{1}{2}\tilde{u}_m(y+1)$ particles from its upper and
lower neighbours, and it sends $\tilde{u}_m(y)$ particles. That is,
\begin{equation*}
\sigma_m(x,y) = \frac{1}{2}\big[f\big(h(x)-y+1\big)+f\big(h(x)-y-1\big) \big]-f\big(h(x)-y\big)=1.
\end{equation*}
\item
The rotors at the vertices $(x,y-1)$, $(x,y)$ and $(x,y+1)$ all point inwards
($\downarrow$). Hence $r(x)\geq y+1$ and, comparing the numbers of incoming and
outgoing particles, we have
\begin{align*}
\sigma_m(x,y) & =  \frac{1}{2}\big[f\big(h(x)-y+1\big)+e\big(r(x)-y+1\big)+f\big(h(x)-y-1\big)\\
            &\phantom{=}\;+e\big(r(x)-y-1\big)\big]-f\big(h(x)-y\big)-e\big(r(x)-y\big)=1.
\end{align*}
\item When the rotors from $1$ to $y-1$ point inwards ($\downarrow$) and
from $y$ to $h(x)$ point outwards ($\uparrow$), then $r(x)=y-1$, and we have
\begin{align*}
\sigma_m(x,y)  &= \dfrac{1}{2}\big[f\big(h(x),r(x),y-1\big)+e(0)-1\big] \\
               &\phantom{=}\;+\frac{1}{2} f\big(h(x),r(x),y+1\big)-f\big(h(x),r(x),y\big)=1.
\end{align*}
\item The last case which can appear is when all rotors from $y$ to $0$
point inwards ($\downarrow$), and from $y+1$ to $h(x)$ outwards ($\uparrow$).
Then $r(x)=y$ and
\begin{align*}
 \sigma_m(x,y) &=  \dfrac{1}{2}\big[f\big(h(x),r(x),y-1\big)+e(1)-1\big] \\
               &\phantom{=}\;+\frac{1}{2} f\big(h(x),r(x),y+1\big)-f\big(h(x),r(x),y\big)-1=1.
\end{align*}
\end{enumerate}
Therefore $\sigma(x,y) = 1$ for all $(x,y)\in B_m$ with $y \geq 2$. Moreover, no closed cycle 
is formed by these vertices in the final rotor configuration $\rho_m$.
%================================================================================
%==============The case y=1======================================================

\textbf{Case 2. $y=1$}:
Consider $\sigma_m(z)$ for the vertex $z = (x,1)$.
For $x \geq 9$ the number of inwards pointing rotors $r(x)$ is always greater than $2$. So with
the exception of a finite number of exceptional points ($x\in\{1,2,5,8\}$), all relevant rotors on the teeth
are pointing inwards ($\downarrow$) and the vertex $z$ receives $\left\lceil\frac{1}{2}u'(x,2)\right\rceil$
particles from its upper neighbour. For $x \geq 3$, the number $u'(x,0)$ is divisible by $4$, so
all neighbours of $(x,0)$ receive exactly the same amount of particles. Hence
\begin{align*}
\sigma_m(x,1 ) &=  \frac{1}{4} u'(x,0) + \frac{1}{2}\big(u'(x,2)+1\big) - u'(x,1)= \frac{1}{4} \big[2f\big(h(x)\big)+2e\big(r(x)\big)-2\big]\\
               &\phantom{=}\;+\frac{1}{2}\big[f\big(h(x)-2\big)+e\big(r(x)-2\big)+1\big]-f\big(h(x)-1\big)-e\big(r(x)-1\big)\\
               &= \frac{1}{4}\big[2h(x)(h(x)+1)+2(2r(x)+1)-2\big]\\
               &\phantom{=}\;+\frac{1}{2}[(h(x)-2)(h(x)-1)+2(r(x)-2)+2]\\
               &\phantom{=}\;-(h(x)-1)h(x)-2(r(x)-1)-1=1.
\end{align*}
if $z$ is non-exceptional. At the exceptional points $z=(x, 1)$, for $x\in\{1,2,5,8\}$, 
the correctness of the function $u'$ can be verified by direct computation.

%==================================================================================
%==============The case y=0========================================================
\textbf{Case 3. $x\not=m$ and $y=0$}: On the $x$-axis, the points $z = (x,0)$
for $x\in\{2,5\}$ are again  exceptional and need to be checked separately. The case
$x=0$ does not have to be checked at all, and $x=1$ has already been checked at the start of the proof.

For $x\notin \{2,5\}$, the vertex $z=(x,0)$ receives particles from $(x-1,0)$, $(x+1,0)$, $(x,1)$, $(x,-1)$.
Here $u'(x-1,0)$ and $u'(x+1, 0)$ are again both divisible by $4$. By symmetry $u'(x,1) = u'(x,-1)$,
and the number of inward pointing arrows $r(x) \geq 1$ in this case, hence $z$ receives $u'(x,1) + 1$
particles from its upper and lower neighbours combined. Thus
\begin{align*}
\begin{split}
\sigma_m(x,0) &= \frac{1}{4} u'(x-1, 0) + \frac{1}{4} u'(x+1, 0) + u'(x,1) + 1 - u'(x,0)\\
 &=\frac{1}{4}\big[2f\big(h(x-1)\big)+2e\big(r(x-1)\big)-2\big]+\frac{1}{4}\big[2f\big(h(x+1)\big)+2e\big(r(x+1)\big)-2\big]\\
              &\phantom{=}\;  +f\big(h(x)-1\big)+e\big(r(x)-1)\big)+1 -\big[2f\big(h(x)\big)+2e\big(r(x)\big)-2\big].
\end{split}
\end{align*}
Using that $f(x)=x(x+1)$ and $e(x)=2x+1$ we get
\begin{align}\label{eq:sigma_x_axis}
\begin{split}
\sigma_m(x,0)&=\frac{1}{2}\big[h^2(x-1)+h^2(x+1)-2h^2(x)\big]+\frac{1}{2}\big[h(x-1)+h(x+1)-6h(x)\big]\\
 &\phantom{=}\;+\big[r(x-1)+r(x+1)-2r(x)\big].
\end{split}
\end{align}
In order to check $\sigma_m(x,0)=1$, we have to substitute in equation \eqref{eq:sigma_x_axis}
the function $h(x) = \left\lfloor\frac{(x+1)^2}{3}\right\rfloor$
and the corresponding branch of the function $r(x)$ given in equation \eqref{eq:def_r_x},
depending on the congruence class mod $3$ of $x$. We have to check all three cases separately.
In all cases $\sigma_m(x,0)=1$ holds.

\textbf{Case 4.} Midpoint $z=(m,0)$: Everything until now was independent of the number of
particles $\abs{B_m}$. Since $u_m$ is created from $u'$ by translation and reflection,
the vertex $z=(m,0)$ after translation corresponds to the origin of the cluster.
At the beginning of the process, $\abs{B_m}$ particles are present at $z$,
so $\sigma_0(z) = \abs{B_m}$. We assume that $m$ is big enough, so that
none of the neighbours of $z$ is an exceptional point.

By symmetry, $z=(m,0)$ receives $\frac{1}{2} u'(m-1, 0)$ particles from its neighbours
on the $x$ axis, and $u'(m, 1) + 1$ particles from its neighbours on the teeth. Hence
\begin{align*}
\begin{split}
\sigma_m(m,0) & = \sigma_0(m,0) + \frac{1}{2} u'(m-1, 0) + u'(m, 1) + 1 - u'(m, 0)\\
              & = \abs{B_m}+\frac{1}{2}\big[2f\big(h(m-1)\big)+2e\big(r(m-1)\big)-2\big]+f\big(h(m)-1\big)\\
 &\phantom{=}\;    +e\big(r(m)-1\big)-2f\big(h(m)\big)-2e\big(r(m)\big)+3.          
\end{split}
\end{align*}
Here one has to check again each congruence class mod $3$ separately.
Substituting the formulas for $\abs{B_m}$ obtained in \eqref{eq:B_m_mod3},
into the previous equation, gives the desired result $\sigma_m(z) = 1$.

Finally, we need to check that the final rotor configuration $\rho_m$ is acyclic.
We work again with shifted coordinates. It is clear from the previous section 
that $\rho_m$ restricted to each ``tooth'' is acyclic. Hence it suffices to check that no 
cycles are created by rotors on the $x$-axis. If $z=(x, 0)$, the odometer $u_m(z)$ is divisible by
$4$, except when $x = 2$. So the rotors at these vertices point in the
same direction as in the initial configuration $\rho_0$. The odometer at the
exceptional point $w = \big(2, 0\big)$ is $u'(w) = 23 \equiv 3 \pmod{4}$
independent of $m$. Hence, this rotor points in the direction of one
``tooth''. If the rotor at position $(2,1)$ points towards the $x$-axis, it
creates a directed cycle. By \eqref{eq:def_r_x}, we have $r(2) = 0$, which means
that all arrows on this ``tooth'' are pointing outwards. Hence the rotor at $w$
does not close a cycle. See Figure \ref{fig:symmetric_config_6_7} for a
visualisation of the rotor configurations under consideration.

Therefore all properties of Theorem \ref{thm:rotor_odometer_properties} are
satisfied and this proves the statement.
\end{proof}

\begin{proof}[Proof of Theorem \ref{thm:rotor_shape}]
In the case $m\leq 2$, the statement of the Theorem follows by
direct calculation of the respective aggregation clusters, see
Figure \ref{fig:symmetric_config_3}.
For $m\geq 3$ it follows from the previous Lemma.
\end{proof}

\section{Harmonic Measure}
\label{sec:harm_measure}
In this section, as a direct application of rotor-router walks, 
we compute the \emph{harmonic measure} of the generic set 
$B_m\subset\C$ defined in \eqref{eq:B_m}.  The harmonic measure of $B_m$
is the hitting distribution of the set $\partial B_m$ for a simple random walk
on $\C$ starting at the origin $o$.

We shall first describe the method for finite subsets $B$ of general graphs $G$,
and then we apply it to the case of the comb $\C$ and subsets $B_m$ of the type
defined in \eqref{eq:B_m}. In Theorem \ref{thm:unif_harm_measure}, we identify 
the shape for which the harmonic measure is uniform. We point out that this shape does not
coincide with the rotor-router aggregation cluster from Theorem \ref{thm:rotor_shape}.
We will also describe the asymptotics of the harmonic measure for the rotor-router
aggregation clusters.

In order to estimate the harmonic measure, we shall use an idea of \textsc{Holroyd
and Propp} \cite{holroyd_propp}, which they used to show a variety of inequalities
concerning rotor-router walks and random walks. The method assigns a weight to the particle
and rotor configuration of a rotor-router process, which is invariant under routing
of particles in the system.

\subsection{Rotor Weights}
\label{subsec:rotor_weights}
Let $G$ be a locally finite and connected graph.
Start with a particle configuration $\sigma_0:G\to\Z$ and a rotor configuration
$\rho_0: G\to G$ such that $\rho_0(x) = x_0$ for all $x\in G$, that is, all
initial rotors point to the first neighbour in the rotor sequence $c(x)$.
We further assume that $\sigma_0$ has finite support, i.e., there are only finitely
many particles in the system, so that we don't need to deal with questions of
convergence. We will route particles in the system, and this gives rise to a
sequence $(\rho_t, \sigma_t)_{t\geq0}$ of particle and rotor configurations at
every time $t$. To each of the possible states $(\rho_t,\sigma_t)$ of the system,
we will assign a weight.

Fix a function $\psi:G\to \R$. We define the \emph{particle weights} at time $t$ to be
\begin{equation}
\label{eq:def_particles_weights}
\mathbf{W_P}(t) = \sum_{x\in G}\sigma_t(x) \psi(x).
\end{equation}
Further define the \emph{rotor weights} of vertices $x\in G$ as
\begin{equation}
\label{eq:def_rotor_weights}
w(x, k) =
\begin{cases}
0, &\text{for } k = 0 \\
w(x, k-1) + \psi(x) - \psi\big(x_{k\bmod d(x)}\big), &\text {for } k > 0,
\end{cases}
\end{equation}
where $x_i$ is the $i$-th neighbour of $x$ in the rotor sequence $c(x)$.
Notice that, for $k \geq d(x)$,
\begin{equation}
\label{eq:rotor_weight_laplace}
w(x,k) = w\big(x,k-d(x)\big) - d(x)\bigtriangleup \psi(x).
\end{equation}
Here $\bigtriangleup \psi(x)$ represents the Laplace operator which is defined as
\begin{equation*}
 \bigtriangleup \psi(x)=\frac{1}{d(x)}\sum_{y\in G:\: y\sim x}\big(\psi(y)-\psi(x)\big).
\end{equation*} 
The total \emph{rotor weights} at time $t$ are given by
\begin{equation*}
\mathbf{W_R}(t) = \sum_{x\in G} w(x,u_t(x)),
\end{equation*}
where $u_t(x)$ is the odometer function of this process, that is, the number of
particles sent out by the vertex $x$ in the first $t$ steps.
Note that $\rho_0$ is chosen in such a way that for all $t\geq 0$ and $x\in G$, if 
$i \equiv u_t(x) \bmod d(x)$, then $x_i = \rho_t(x)$.

It is easy to check that the sum of particle and rotor weights are invariant under
routing of particles, i.e., for all times $t,t'\geq 0$
\begin{equation}
\label{eq:rotor_weights_invariance}
\mathbf{W_P}(t) + \mathbf{W_R}(t) = \mathbf{W_P}(t') + \mathbf{W_R}(t').
\end{equation}

\subsection{Harmonic Measure for finite subsets of graphs}
\label{subsec:harm_measure}
As before, let $G$ be a locally finite, connected graph, and
let $B$ be some finite subset of $G$. Write
\begin{equation*}
\partial B = \big\{x\in B: \exists y\not\in B \text{ with } x\sim y \big\}
\end{equation*}
for the \emph{inner boundary} of $B$, and $B^\circ = B \setminus \partial B$. 
The vertices of $\partial B$ will represent the \emph{sink} $S$.

Similarly to Definition \ref{def:rotor_toppling_operator} of Section \ref{sec:rotor-router}, we
define the \emph{particle addition operator} $E_x$, for each vertex $x\in B^\circ$, as follows:
for a rotor configuration $\rho$, let
\begin{equation*}
E_x(\rho) = \rho',
\end{equation*}
where $\rho'$ is the rotor configuration obtained from $\rho$ by adding a new particle at
vertex $x$, and letting it perform a rotor-router walk until the particle reaches a 
vertex in $\partial B$ for the first time. 
By the abelian property of rotor-router walks the operators $E_x$ commute, and they can be
used to define an abelian group, see \cite{chip_rotor_2008} for details and
\cite[Lemma 3.10]{chip_rotor_2008} for the proof of the following statement.
\begin{lem}\label{lem:part_add_op}
The particle addition operator $E_x$ is a permutation on the set of acyclic rotor configurations
on $B^\circ$.
\end{lem}
The \emph{rotor-router group} of $B^\circ$ is defined as the subgroup of permutations of oriented
spanning trees rooted at the sink (that is, acyclic rotor configurations) generated by
$\big\{E_x:\: x\in B^\circ\big\}$. For every finite graph $B^\circ$ the rotor-router group is a
finite abelian group, which is isomorphic to the \emph{abelian sandpile group}. See once again
\cite{chip_rotor_2008} for details.

Consider the \emph{simple random walk $(X_t)_{t\geq 0}$} on $G$, i.e., a Markov chain
with state space $G$, and transition probabilities given by
\begin{equation*}
 p(x,y)=\frac{1}{d(x)},\quad \text{ for all } x, y\in G, \text{ with } x\sim y.
\end{equation*}
Then $X_t$ is a $G$-valued random variable, and represents the random 
position of the random walker at the discrete time $t$. For given 
$x\in G$, we write $\mathbb{P}_x$ for the law of a random walk starting at $x$.
Consider the \emph{stopping time}
\begin{equation*}
T = \inf \big\{t\geq 0: X_t \in \partial B\big\}.
\end{equation*}
For $z\in \partial B$, let 
\begin{equation*}
\nu_x(z) = \P_x[X_T = z], 
\end{equation*}
be the \emph{harmonic measure} at $z$ with starting point $x$, that is, the probability 
that a random walk starting at $x$ hits $\partial B$ for the first time in $z$.

Take the harmonic measure itself as the weight function. More explicitly,
fix a vertex $z\in\partial B$, and define the weight function $\psi(x)$ as
\begin{equation*}
\psi(x) = \psi_z(x) = \nu_x(z).
\end{equation*}
Let us define the following process. Start with $n$ particles at the origin $o$,
and an arbitrary acyclic rotor configuration $\rho_0$. Let the particles
perform rotor-router walks until they reach a vertex in $\partial B$ for the first
time, where they stop. Denote by $t^\star = t^\star(n)$ the number of steps this
process takes to complete, and for each $w\in\partial B$, write $e(w)$ for the
number of particles that stopped in $w$ at the end of this procedure. We denote by $\hat u$ 
the normalized rotor-router odometer function of this process, that is, for all $x\in B$, 
\begin{equation*}
\hat u(x) = \frac{\text{number of particles sent out by $x$}}{d(x)}.
\end{equation*}
Using the invariance of the sum of rotor and particle weights under rotor-router
walks, as in \eqref{eq:rotor_weights_invariance}, we get
\begin{equation}
\label{eq:invariance_harmonic}
n \psi(o) = \sum_{w\in\partial B} e(w) \psi(w) + \mathbf{W_R}(t^\star),
\end{equation}
since $\mathbf{W_R}(0)=0$, $\mathbf{W_P}(0)=n \psi(0)$ and 
$\mathbf{W_P}(t^\star)=\sum_{w\in\partial B} e(w) \psi(w)$.
Equation \eqref{eq:invariance_harmonic} reduces to
\begin{equation}
\label{eq:harmonic_measure_weight_eq}
n \psi(o) = e(z) + \mathbf{W_R}(t^\star),
\end{equation}
because $\psi(w) = \nu_w(z) = \delta_w(z)$, if $w\in\partial B$.

The initial rotor configuration $\rho_0$ is chosen to be acyclic. Therefore, there exists a number $n$
such that, after all $n$ particles performed their rotor-router walks, all rotors in $B^\circ$ made only full
turns, i.e., $\rho_0 = \rho_t^\star$. This claim follows from Lemma \ref{lem:part_add_op}. 
Hence, $n$ is a multiple of the order of $E_o$ in the rotor-router group. Since $\psi$ is a 
harmonic function on $B^\circ$, using a $n$ with the above property gives
$\mathbf{W_R}(t^\star)=0$, which together with \eqref{eq:harmonic_measure_weight_eq} leads to 
\begin{equation}
\label{eq:harmonic_measure_proportional}
n\cdot\nu_o(z) = e(z).
\end{equation}
Thus, the harmonic measure $\nu_o$ of $B$ is proportional to the number of particles which stopped
at $\partial B$. While a number $n$ with the right property is difficult to calculate, we can still use
equation \eqref{eq:harmonic_measure_proportional} in order to derive asymptotics of the harmonic
measure of subsets $B_m$ of the comb $\C$, and in some cases even to calculate it explicitly.

\subsection{Subsets of the Comb}
Let us consider subsets $B_m$ of $\C$, of the type defined in \eqref{eq:B_m}, with generic 
positive function $h: \N_0 \to \N_0$. Recall here the definition of $B_m$.
\begin{equation*}
B_m = \big\lbrace (x,y)\in\C:\: \abs{x} \leq m, \abs{y} \leq h(m-\abs{x})\big\rbrace\quad\text{for }m \in\N.
\end{equation*}
By construction, all rotors make only full turns if we perform the rotor-router process 
from Section \ref{subsec:harm_measure}, for the set $B_m$. This implies that the corresponding
normalized odometer function $\hat u$ is harmonic outside the origin and its Laplacian is given by
\begin{equation}
\label{eq:dirichlet_hm}
\bigtriangleup \hat u(w) =
\begin{cases}
0, \quad & w\in B_m\setminus \big(\partial B_m \cup \{o\}\big) \\
-n,      & w = o,
\end{cases}
\end{equation}
and $\hat u(w) = 0$, for $w \in \partial B_m$. Write $\nu_{m,o}(w)$ for the harmonic measure of $B_m$ and $e_m(w)$ for the
number of rotor-router particles stopped in $w\in\partial B_m$. By symmetry of the set $B_m$, it is clear 
that also $e_m(w)$ and $\nu_{m,o}(w)$ are symmetric. More precisely,
if $w = (x, y)$ and $w' = (\abs{x}, \abs{y})$ then
\begin{equation*}
e_m(w) = e_m(w') \qquad\text{and}\qquad \nu_{m,o}(w) = \nu_{m,o}(w').
\end{equation*}
Hence it is enough to work in one quadrant. We will choose the second quadrant (i.e. $x\leq 0$ and
$y\geq 0$), and for simplicity of notation shift the set $B_m$ by $m$ in the direction of the positive $x$ axis, such
that the leftmost point of $B_m$ has coordinate $(0,0)$ and its center $o$ has coordinate $(m,0)$.
So, the set under consideration is now
\begin{equation*}
B^\Box_m=\big\{(x,y)\in \C:\: 0\leq x\leq m,\, 0\leq y \leq h(x)\big\} 
\end{equation*}
We will also use $\hat u$ for the normalized odometer function and $e_m$ for the number of 
particles which hit boundary points  in the shifted coordinate system. Additionally, since 
$e_m$ is defined only on $\partial B_m$ we write $e_m(x) = e_m(x, h(x))$, for $0\leq x \leq m$.

Solving the Dirichlet problem \eqref{eq:dirichlet_hm} on the ``teeth'' of the comb,
gives for $(x,y) \in B^\Box_m$,
\begin{equation}
\label{eq:hm_odo}
\hat u(x,y) = e_m(x)\cdot\big(h(x) - y\big).
\end{equation}
On the $x$ axis, for $(x, 0) \not= o$, the harmonicity gives
\begin{equation*}
\hat u(x+1,0) + \hat u(x-1, 0) + 2 \hat u(x,1) = 4 \hat u(x,0),
\end{equation*}
which together with \eqref{eq:hm_odo} leads to the following recursion for $e_m(x)$ and $0<x<m$:
\begin{equation}\label{eq:hm_rec}
e_m(x+1)h(x+1) + e_m(x-1)h(x-1) - 2e_m(x)\big(h(x) + 1\big) = 0.
\end{equation}

We are now ready prove Theorem \ref{thm:unif_harm_measure}, which characterizes sets of uniform harmonic measure.
\begin{proof}[Proof of Theorem \ref{thm:unif_harm_measure}]
Let $B_m$ be defined as in \eqref{eq:B_m}, with $h(x)=x^2$. 
From \eqref{eq:hm_rec} we get the recursion
\begin{equation}
\label{eq:hm_rec_x2}
e_m(x+1)(x+1)^2 + e_m(x-1)(x-1)^2 - 2 e_m(x)(x^2+1) = 0, \quad\text{for } 0 < x < m.
\end{equation}
Since $h(1) = 1$, the vertex $z=(1,0)$ in the shifted coordinate
system has three neighbours on the boundary $\partial B_m$.
By construction, the rotor at $(1,0)$ makes a number of full turns, 
hence all of these three neighbours receive the same amount of particles from $z$.
Therefore $e_m(0) = e_m(1)$.
By induction, it is easy to see that the sequence $e_m(x)$ is constant. 
Assuming $e_m(x-1) = e_m(x)$, the recursion \eqref{eq:hm_rec_x2} reduces to
\begin{equation*}
e_m(x+1)(x+1)^2 - e_m(x)(x+1)^2 = 0.
\end{equation*}
which implies that $e_m(x+1) = e_m(x)$. Because $e_m(x)$ is by construction proportional
to the harmonic measure $\nu_{m,o}$, we get the claim.
\end{proof}

In general we can compute the harmonic measure for all sets $B_m$, where \eqref{eq:hm_rec_bm} can
be solved exclicitly.

\subsection{Harmonic Measure of the Rotor-Router Cluster}
\label{subsec:harmonic_measure_bm}
This section is dedicated to the proof of Theorem \ref{lem:harmonic_measure}.
Recall that the rotor-router cluster obtained in Theorem \ref{thm:rotor_shape}
is a set of type $B_m$, as defined in \eqref{eq:B_m}, with
$h(x) = \left\lfloor {(x+1)^2}/3 \right\rfloor$.

Like before, let $e_m(x)$ be the number of particles stopped at boundary points 
$(x,h(x))\in\partial B_m$ in the rotor-router process defined in Section \ref{subsec:harm_measure}.
By linearity also the normalized sequence $e(x) = \frac{e_m(x)}{e_m(0)}$ is a solution of
the recurrence \eqref{eq:hm_rec}, and since $h(1) = 1$ we have $e(0)=e(1)=1$. Hence the function $e(x)$
is independent of $m$.
Rewriting \eqref{eq:hm_rec} in this case, we get a linear recurrence with non-polynomial coefficients.
While an explicit answer is not feasible, we can derive asymptotics of the special solution $e(x)$,
by converting the recurrence into an equivalent system of linear differential equations. 
We will prove Theorem \ref{lem:harmonic_measure} by showing that the function $e(x)$ has linear growth,
which is accomplished in the next two Lemmas.

\begin{lem}\label{lem:most_lin}
There exists a constant $c<\frac{1}{2}$ such that
\begin{equation*}
 \frac{e(x)}{x}\to c \text{ as } x\to \infty.
\end{equation*}
\end{lem}

\begin{proof}
Substitute $\tilde{e}(x) = \frac{e(x)}{x}$ for $x>0$, which transforms \eqref{eq:hm_rec} into
\begin{equation}
\label{eq:hm_rec_bm}
\tilde{e}(x-1)(x-1)h(x-1) + \tilde{e}(x+1)(x+1)h(x+1) - 2\tilde{e}(x) x \big( h(x) + 1 \big) = 0.
\end{equation}
The sequence $\tilde{e}(x)$ converges if and only if $e(x)$ grows at most linearly. 
Since $e(x)$ is positive by construction, it suffices to check that $\tilde{e}(x)$ is decreasing. 
For this, consider the auxiliary function $h'(x) = \frac{(x+1)^2}{3} - \frac{1}{3}$. 
We have to distinguish three cases
\begin{equation*}
h(x) =
\begin{cases}
h'(x), \quad & x \equiv 0\mod{3} \\
h'(x), \quad & x \equiv 1\mod{3} \\
h'(x) + \frac{1}{3}, \quad & x \equiv 2\mod{3} \\
\end{cases}
\end{equation*}
We prove the monotonicity of $\tilde{e}(x)$ by induction. 
Assuming $\tilde{e}(x) < \tilde{e}(x-1)$
for $x \equiv 0\mod{3}$ we show that $\tilde{e}(x+3) < \tilde{e}(x+2) < \tilde{e}(x+1) < \tilde{e}(x)$.
The induction base follows by calculating the first elements of the sequence.

\textbf{Case 1.}
Assume $x \equiv 0\mod{3}$ and $\tilde{e}(x) < \tilde{e}(x-1)$. Then \eqref{eq:hm_rec_bm} 
can be rewritten as
\begin{equation*}
\tilde{e}(x+1)(x+1)h'(x+1) = 2\tilde{e}(x) x \big(h'(x) + 1\big) -
                             \tilde{e}(x-1)(x-1)\left(h'(x-1) + \tfrac{1}{3}\right).
\end{equation*}
Using the induction hypothesis and the definition of $h'(x)$, we get
\begin{equation}
\label{eq:hm_bm0}
\tilde{e}(x+1) < \tilde{f}_0(x) \cdot \tilde{e}(x),
\end{equation}
with
$\tilde{f}_0(x) =\frac{x^2 + 2x}{x^2 + 2x + 1} < 1$,
which implies $\tilde{e}(x+1) < \tilde{e}(x)$.

\textbf{Case 2.} Assume $x \equiv 1\mod{3}$ and  $\tilde{e}(x) < \tilde{f}_0(x-1) \cdot \tilde{e}(x-1)$. 
Like before, rewrite \eqref{eq:hm_rec_bm} as
\begin{equation*}
\tilde{e}(x+1)(x+1)\big(h'(x+1) + \tfrac{1}{3}\big) = 2\tilde{e}(x) x \big(h'(x) + 1\big) -
                                                      \tilde{e}(x-1)(x-1) h'(x-1).
\end{equation*}
This gives, by \eqref{eq:hm_bm0} 
\begin{equation}
\label{eq:hm_bm1}
\tilde{e}(x+1) < \tilde{f}_1(x) \cdot \tilde{e}(x),
\end{equation}
for 
\begin{align*}
\tilde{f}_1(x) &= \frac{2x\big(h'(x) + 1\big) - \tilde{f}_0(x-1)^{-1}(x-1)
                   h'(x-1)}{(x+1)\big(h'(x+1)+\tfrac{1}{3}\big)} \\
               &= \frac{x^2 + 3x}{x^2 + 3x + 2} < 1,
\end{align*}
which implies $\tilde{e}(x+1) < \tilde{e}(x)$.

\textbf{Case 3.} Finally, assuming $x \equiv 2\mod{3}$ and
 $\tilde{e}(x) < \tilde{f}_1(x-1) \cdot \tilde{e}(x-1)$, we get
\begin{equation*}
\tilde{e}(x+1)(x+1)h'(x+1) = 2\tilde{e}(x) x \big(h'(x) + \tfrac{4}{3}\big) -
                             \tilde{e}(x-1)(x-1) h'(x-1).
\end{equation*}
Applying \eqref{eq:hm_bm1}, we obtain
\begin{equation}
\label{eq:hm_bm2}
\tilde{e}(x+1) < \tilde{f}_2(x) \cdot \tilde{e}(x),
\end{equation}
for the function $\tilde{f}_2(x) =  \frac{x^4 + 7 x^3 + 17 x^2 + 17 x}{x^4 + 7 x^3 + 17 x^2 + 17 x + 6} < 1$.

This shows that $\tilde{e}(x)$ is decreasing and therefore convergent, which also means
that there exists a constant $c$ such that $\frac{e(x)}{x}\to c$, as $x\to \infty$.
The fact that $c < \frac{1}{2}$ follows by computing the first few values of the sequence $e(x)$, 
using $e(0)=e(1)=1$ as starting values in the recursion \eqref{eq:hm_rec}.
By monotonicity we then get $\tilde{e}(x) < \frac{1}{2}$ for all $x \geq 20$.
\end{proof}

The next result shows that the function $e(x)$ has at least linear growth.
\begin{lem}\label{lem:least_lin}
The constant $c$ in Lemma \ref{lem:most_lin} is strictly positive.
\end{lem} 
\begin{proof}
To show that $c>0$, we use singularity analysis of linear differential equations. 
For this, we split $e(x)$ into three sequences modulo $3$, i.e., for $k\in\N$ write
\begin{align*}
e_i(k) &= e(3k+i) \quad \text{for } i=0,1,2,
\end{align*}
and rewrite \eqref{eq:hm_rec} for each congruence
class of $x \bmod{3}$ in terms of $k$. This leads to a system of linear recursions which
can be written in matrix form as
\begin{equation}
\label{eq:hm_bm_matrix_rec}
A_k\cdot \vec{e}(k-1) = B_k\cdot \vec{e}(k),
\end{equation}
with $\vec{e}(k) = \big(e_0(k), e_1(k), e_2(k)\big)^t$, and the matrices $A_k$ and $B_k$ given as
\scriptsize
\begin{align*}
A_k = \begin{pmatrix}
0 & 3 k^2 - 2 k & -6 k^2 - 2\\[0.5em]
0 & 0 & 3 k^2 \\[0.5em]
0 & 0 & 0
\end{pmatrix},
\quad
B_k = \begin{pmatrix}
-3 k^2 - 2k & 0 & 0 \\[0.5em]
6 k^2 + 4 k + 2 & -3 k^2 - 4 k - 1 & 0 \\[0.5em]
3 k^2 + 2 k & -6 k^2 - 8 k - 4 & 3 k^2 + 6 k +3
\end{pmatrix}.
\end{align*}
\normalsize
The initial values are given by $\vec{e}(0) = \big(1,1,\tfrac{4}{3}\big)^t$.
Denote by $E_i(z) = \sum_{k\geq 0}e_i(k) z^k$ the generating function of $e_i(k)$, $i=0,1,2$.
Using the identities
\begin{equation*}
\sum_{k\geq 0}k e_i(k)z^k = z\frac{\partial}{\partial z} E_i(z)\quad \text{and}\quad
\sum_{k\geq 0}k^2 e_i(k)z^k = z^2\frac{\partial^2}{\partial z^2} E_i(z) + z\frac{\partial}{\partial z} E_i(z),
\end{equation*}
the matrix recursion \eqref{eq:hm_bm_matrix_rec} can be transformed into the
following system of linear differential equations for the generating functions $E_i(k)$
\begin{equation}
\label{eq:hm_bm_diffeq}
C\cdot \vec{E}(z) = b,
\end{equation}
where $\vec{E}(z) = \big(E_0(z), E_1(z), E_2(z)\big)^t$, and $C$ is a matrix of linear differential
operators given as
\scriptsize
\begin{align*}
C = \begin{pmatrix}
5 \p + 3 z \pp & 1 + 7 z \p + 3 z^2\pp & -8 -18 z\p -6 z^2\pp \\[1.5em]
-2 -10 z\p -6 z^2 \pp & 1 + 7 z \p + 3 z^2\pp & 3 z + 9 z^2 \p + 3 z^3 \pp \\[1.5em]
5 z\p + 3 z^2 \pp & -4 -14 z\p-6 z^2\pp & 3 + 9 z \p + 3 z^2 \pp
\end{pmatrix},
\quad 
b = \begin{pmatrix}
0 \\
e_1(0) - 2 e_0(0) \\
0
\end{pmatrix}.
\end{align*}
\normalsize
To solve \eqref{eq:hm_bm_diffeq} asymptotically, we consider $C$ as a matrix with entries in the Weyl
algebra, that is, the noncommutative ring of linear differential operators with polynomial coefficients,
see \cite{lam_1991}. We can perform a division-free Gauss elimination over this ring to transform $C$
into row echelon form, which gives a single differential equation only involving $E_2(z)$. The actual
computations were performed using the computer algebra system
\texttt{FriCAS}\footnote{\url{http://fricas.sourceforge.net}}. The result is a differential equation of
order 7 for $E_2(z)$:

\begin{align}
\label{eq:hm_bm_diffeq_e2}
\begin{split}
{\textstyle
\frac{81}{8} (z + 2) (z - 1)^5 z^6
    \frac{\partial^7}{\partial z^7}}\\[0.3em]
{\textstyle
+ \frac{1269}{4} (z - 1)^4 z^5 \left(z^2 + z - \frac{76}{47}\right)
    \frac{\partial^6}{\partial z^6}}\\[0.3em]
{\textstyle
+ \frac{27531}{8} (z - 1)^3 z^4 \left(z^3 - \frac{24}{437} z^2 - \frac{7149}{3059} z + \frac{3826}{3059}\right) 
    \frac{\partial^5}{\partial z^5}}\\[0.3em]
{\textstyle
+ \frac{127725}{8} \cdot (z - 1)^2 \cdot z^3 \cdot \left(z^4 - \frac{50039}{42575} z^3 - \frac{82401}{42575} z^2 + \frac{132307}{42575} z - \frac{38554}{42575}\right)
    \frac{\partial^4}{\partial z^4}}\\[0.3em]
{\textstyle
+ 31785 (z - 1) z^2 \left(z^5 - \frac{100697}{42380} z^4 - \frac{1164}{10595} z^3 + \frac{36215}{8476} z^2 - \frac{5651}{1630} z + \frac{6234}{10595}\right)
\frac{\partial^3}{\partial z^3}}\\[0.3em]
{\textstyle
+ 23970 z \cdot \left(z^6 - \frac{117579}{31960} z^5 + \frac{114057}{31960} z^4 + \frac{15053}{6392} z^3 - \frac{208329}{31960} z^2 + \frac{59229}{15980} z - \frac{1243}{3995}\right)
    \frac{\partial^2}{\partial z^2}}\\[0.3em]
{\textstyle
+ 4935 \left(z^6 - \frac{1354}{329} z^5 + \frac{1843}{329} z^4 - \frac{4479}{3290} z^3 - \frac{12209}{3290} z^2 + \frac{1466}{329} z - \frac{32}{329}\right)
    \frac{\partial}{\partial z}}\\[0.3em]
{\textstyle
+ 105 \left(z^5 - \frac{494}{105} z^4 + \frac{881}{105} z^3 - \frac{201}{70} z^2 + \frac{4411}{210} z + \frac{1006}{105}\right) = 0}
\end{split}
\end{align}
Using singularity analysis for linear differential equations, we can derive asympotics of $e_2(k)$. See \textsc{Flajolet and Sedgewick} \cite[Theorem VII.10]{flajolet_sedgewick} for details.
The coefficient of the highest order term $\frac{\partial^7}{\partial z^7}$ is given by
\begin{equation*}
\frac{81}{8} (z + 2) (z - 1)^5 z^6,
\end{equation*}
hence the \emph{dominant non-zero singularity} $\xi$ is equal to $1$. 
Since all coefficients in \eqref{eq:hm_bm_diffeq_e2} are
given in factorized form, it is immediate that $\xi$ is a \emph{regular singularity}.
Calculating the \emph{indicial polynomial} for the singularity $\xi$ gives
\begin{equation*}
I_\xi(\theta) = \theta^{7} - 17 \theta^{6} + 99 \theta^{5} - 187 \theta^{4} - 220 \theta^{3} + 1044 \theta^{2} - 720 \theta.
\end{equation*}
For the definition of a regular singularity and the indicial polynomial,
see once again \textsc{Flajolet and Sedgewick} \cite[Chapter VII.9]{flajolet_sedgewick}.
The roots of $I_\xi(\theta)$ are $-2, 0, 1, 3, 4, 5$ and $6$. Since they differ by
integers, the asymptotics of $e_2(k)$ is given by
\begin{equation*}
e_2(k) \sim c\cdot \xi^{-k} k^\beta \log^l k,
\end{equation*}
where $l$ is an integer and $\beta$ is the biggest solution of the equation 
$I(-\beta -1) = 0$, see \cite[page 521, equation 118]{flajolet_sedgewick}.
The $\sim$ sign means ``approximately equal'' (in the precise sense that the ratio
of both terms tends to $1$ as $k$ gets large).

In our case $\beta=1$, and we have
\begin{equation}
\label{eq:constant_l}
e_2(k) \sim c\cdot k \log^l k,
\end{equation}
and this proves the desired.
\end{proof}
While it is not known how to calculate the constant $l$ in \eqref{eq:constant_l}
in the general case, from Lemma \ref{lem:most_lin} we already
know that $e_2(k)$ grows at most linearly, hence $l=0$. Therefore, Lemma
\ref{lem:most_lin} and Lemma \ref{lem:least_lin} together imply Theorem
\ref{lem:harmonic_measure}.

\textbf{Acknowledgements:} We are grateful to Franz Lehner for interesting discussions 
and for helping with the computer computations in the proof of Lemma \ref{lem:least_lin},
and also to the anonymous referee whose comments led to an essential improvement of the paper.

The research of Wilfried Huss and Ecaterina Sava was partially supported by the 
Austrian Science Fund (FWF): W1230-N13.

\newcommand{\etalchar}[1]{$^{#1}$}

\end{document}